\documentclass[12pt]{amsart} 
\usepackage[a4paper,margin=1in,footskip=0.25in]{geometry}
\usepackage{mathrsfs}
\usepackage{amssymb}
\usepackage{amsmath}
\usepackage{amsthm}
\usepackage{thmtools} 
\usepackage{mathtools}
\usepackage{nicematrix}
\usepackage{tikz-cd}
\usepackage{enumitem}
\usepackage{subfiles}
\usepackage{faktor}
\usepackage{quiver} 
\usepackage{subcaption}
\usepackage{float}
\usepackage{xcolor}
\usepackage[colorinlistoftodos]{todonotes}
\usetikzlibrary{shapes.geometric,decorations.pathreplacing,decorations.markings,decorations.pathmorphing,calc,patterns,hobby}

\PassOptionsToPackage{pdfusetitle,colorlinks}{hyperref}
\usepackage{bookmark}
\hypersetup{
  linkcolor={blue!70!cyan!80!black},
  citecolor={blue!30!cyan!80!black},
  urlcolor={blue!50!cyan!80!black}
}

\usepackage{cleveref} 

\newtheorem{theorem}{Theorem}[section]
\newtheorem*{theorem*}{Theorem}
\newtheorem{lemma}[theorem]{Lemma}%
\newtheorem{proposition}[theorem]{Proposition}%

\theoremstyle{definition}
\newtheorem{definition}[theorem]{Definition}

\theoremstyle{remark}
\newtheorem{remark}[theorem]{Remark}
\newtheorem{example}[theorem]{Example}

\newcommand{\sB}{\mathcal{B}}
\newcommand{\sC}{\mathcal{C}}
\newcommand{\sD}{\mathcal{D}}

\newcommand{\sF}{\mathcal{F}}
\newcommand{\sG}{\mathcal{G}}

\newcommand{\sK}{\mathcal{K}}

\newcommand{\sO}{\mathcal{O}}

\newcommand{\sU}{\mathcal{U}}
\newcommand{\sV}{\mathcal{V}}
\newcommand{\sW}{\mathcal{W}}

\newcommand{\bbA}{\mathbb{A}}

\newcommand{\bbN}{\mathbb{N}}

\newcommand{\bbZ}{\mathbb{Z}}

\newcommand{\frm}{\mathfrak{m}}

\DeclareMathOperator{\coker}{coker}
\DeclareMathOperator{\im}{im}

\DeclareMathOperator{\rk}{rk}

\DeclareMathOperator{\id}{id}
\DeclareMathOperator{\Hom}{Hom}
\DeclareMathOperator{\End}{End}
\DeclareMathOperator{\Aut}{Aut}

\DeclareMathOperator{\proj}{proj}
\DeclareMathOperator{\inj}{inj}
\DeclareMathOperator{\colim}{colim}

\DeclareMathOperator{\AffSch}{AffSch}
\DeclareMathOperator{\Sets}{Sets}

\DeclareMathOperator{\rep}{rep}
\DeclareMathOperator{\GL}{GL}
\DeclareMathOperator{\Mat}{Mat}
\DeclareMathOperator{\rad}{rad}
\DeclareMathOperator{\diag}{diag}
\DeclareMathOperator{\Ext}{Ext}

\newcommand{\GLdd}{\GL_\mathbf{d}}
\newcommand{\leqdeg}{\leq_{\textnormal{deg}}}

\newcommand{\leqdeghom}{\leq_{\textnormal{deg}}^{\sK}}
\newcommand{\dd}{\mathbf{d}}

\newcommand{\leqext}{\leq_{\textnormal{ext}}}
\newcommand{\leqhom}{\leq_{\textnormal{hom}}}
\newcommand{\leqexthom}{\leq_{\textnormal{ext}}^{\sK}}
\newcommand{\leqhomhom}{\leq_{\textnormal{hom}}^{\sK}}

\newcommand{\PP}{\textbf{P}}

\newcommand{\GP}{G_\PP}
\newcommand{\varP}{\sC_\PP}
\newcommand{\ivarPhom}[2]{\sK^{[#1,#2]}_g}
\newcommand{\bvarPhom}{\sK^{\mathbf{b}}_g}
\newcommand{\varPhom}{\sK^{\ast}_g}
\newcommand{\iGhom}[2]{\sG^{[#1,#2]}_g}
\newcommand{\bGhom}{\sG^{\mathbf{b}}_g}
\newcommand{\Ghom}{\sG^{\ast}_g}

\newcommand{\chninj}{\sC^{[-n,0]}(\inj \Lambda)}
\newcommand{\chnproj}{\sC^{[-n,0]}(\proj \Lambda)}
\newcommand{\chnmod}{\sC^{[-n,0]}(\mod \Lambda)}

\newcommand{\bchnproj}{\sC^{\mathbf{b}}(\proj \Lambda)}
\newcommand{\bchnmod}{\sC^{\mathbf{b}}(\mod \Lambda)}

\newcommand{\ichnproj}[2]{\sC^{[#1,#2]}(\proj \Lambda)}
\newcommand{\ichnmod}[2]{\sC^{[#1,#2]}(\mod \Lambda)}

\newcommand{\homprojstar}{\sK^{\ast}(\proj \Lambda)}
\newcommand{\homprojint}[2]{\sK^{[#1,#2]}(\proj \Lambda)}

\newcommand{\bhomproj}{\sK^{\mathbf{b}}(\proj \Lambda)}

\let\mod\relax
\DeclareMathOperator{\mod}{mod}

\let\amsamp=&
\newcommand{\pmatquiver}[1]{%
  \left(\begin{smallmatrix}#1\end{smallmatrix}\right)%
}

\newcommand{\pmat}[1]{%
\begin{pNiceMatrix}
[small]
#1
\end{pNiceMatrix}
}

\newcommand{\biblio}{\bibliographystyle{amsalpha}\bibliography{literature}}

\begin{document}

\renewcommand{\biblio}{}

\title{Degenerations of chain complexes}

\author{Judith Marquardt}
\address[Judith Marquardt]{Univ. Grenoble Alpes, CNRS, IF, 38000 Grenoble, France, and Université Paris-Saclay, UVSQ, CNRS, Laboratoire de Mathématiques de Versailles, 78000, Versailles, France}
\email{\href{mailto:judith.marquardt@univ-grenoble-alpes.fr}{judith.marquardt@univ-grenoble-alpes.fr}}

\keywords{}
\thanks{}

\begin{abstract}
We study degenerations of orbits in varieties of chain complexes of projective modules. We show that degenerations are characterized in terms of certain admissible short exact sequences in the exact category of complexes of projectives, analogously to work from Riedtmann and Zwara.
We then extend our study to the homotopy category. In this context, complexes of projectives with a given $g$-vector form an ind-variety, and we prove that degenerations of orbits are characterized by the existence of certain distinguished triangles in the category. This links to the algebraic definition of degeneration for triangulated categories introduced by Jensen, Su and Zimmermann.
\end{abstract}

\maketitle

\tableofcontents

\biblio

\section{Introduction}
\label{sect:introduction}

Let $k$ be an algebraically closed field and $\Lambda$ a finite dimensional $k$-algebra given by a quiver $Q$ and an admissible ideal $I$. In the study of representation theory, it has proven interesting to consider modules as points in a variety, see for example \cite{G,ADF,Rie,Z,CBS}.

For a fixed dimension vector $\dd = (d_v)_{v \in Q_0}$, the modules of this dimension are the closed points in the representation variety
\begin{equation*}
    \rep(\Lambda,\dd) \subset \prod_{a \in Q_1} \Mat_{d_{t(a) \times s(a)}}(k).
\end{equation*}
The group $\GLdd = \prod_{v \in Q_0} \GL_{d_v}$ acts on $\rep(\Lambda,\dd)$ via base change. Then the orbit of a module $M$, denoted by $\sO_M$, is its isomorphism class. We say that a module $N$ is a degeneration of a module $M$, written $M \leqdeg N$, if there is an inclusion $\sO_N \subset \overline{\sO_M}$. This forms a partial order on the isomorphism classes of modules of a given dimension vector. It is bracketed by two other partial orders: the ext and hom order \cite{ADF, Rie, B2}. The ext order is defined as the transitive closure of $Y \leqext X \oplus Z$ where there is a short exact sequence $0 \rightarrow X \rightarrow Y \rightarrow Z \rightarrow 0$. The hom order is defined as $M \leqhom N$ if for any finite dimensional module $X$, $\dim_k \Hom(X,M) \leq \dim_k \Hom(X,N)$. It is known that the hom order is a refinement of the degeneration order which is in turn a refinement of the ext order, i.e.  there is a chain of implications $(M \leqext N) \Rightarrow (M \leqdeg N) \Rightarrow (M \leqhom N)$.

Degenerations have been classified in representation theoretical terms in \cite{Rie,Z}: for two modules $M$ and $N$, there exists a degeneration $M \leqdeg N$ if and only if there is a short exact sequence of the form $0 \rightarrow N \rightarrow M \oplus Z \rightarrow Z \rightarrow 0$ in $\mod \Lambda$. 

This result has opened up the study of degenerations to other algebraically inspired varieties, for example for infinite dimensional algebras \cite{Y}, in Nakajima's graded affine quiver varieties \cite{KS} and for the stable category of Cohen-Macaulay modules \cite{Y_st, Hi}. It also has inspired an algebraic definition of degeneration in triangulated categories where by definition $X \leqdeg Y$ if there is a triangle of the form $Y \rightarrow X \oplus Z \rightarrow Z \rightarrow X[1]$, see \cite{JSZ, SZ_sym, SZ_deg_0}. If a distinction is necessary, we call these new degenerations algebraic degenerations and those obtained through orbit closure inclusion geometric degenerations.

In this context, varieties associated to the derived category have been studied. In \cite{JSZ_ch_cx}, objects in $\sD^b(A)$ are represented by right bounded chain complexes on projective modules. A topological space is built representing such complexes over the same fixed underlying modules. Any two objects can be lifted up to homotopy into the same topological space where degenerations are in a one to one correspondence with triangles of the above form. In \cite{JMS}, the derived category is approached via $A_\infty$-modules with a variety for the different homological dimensions. The authors show that geometric degenerations give rise to algebraic degenerations. However not all algebraic degenerations are obtained this way. 

This paper aims to offer a new approach to the study of geometric and algebraic degenerations of chain complexes and chain complexes up to homotopy. We define geometric spaces in which points are complexes. These spaces admit a group action whose orbits correspond to the isomorphism class of complexes in either $\bchnproj$ or $\bhomproj$. Then we show that geometric degenerations are in a one to one correspondence to algebraic degenerations in the appropriate category.

In the first part of the paper, we study the varieties with group action introduced by \cite{JSZ_ch_cx} in the case that the chain complexes are bounded. The dimension vector $\dd$ is replaced by the underlying modules of a complex $\PP = (P^i)_{i \in \bbZ}$. Then the complexes with these modules can be seen as closed points in a closed subvariety
\begin{equation*}
    \varP \subset \prod_{i \in \bbZ} \Hom(P^i, P^{i+1}).
\end{equation*}
The group $\GP = \prod_{i \in \bbZ} \Aut(P^i)$ acts on $\varP$ by conjugation such that the orbits are the isomorphism classes of complexes. We adapt the strategy of the proofs of \cite{Rie,Z} showing the following.
\begin{theorem*}[\Cref{thm:deg_and_ses_in_chnproj}]
    Let $M,N \in \varP$. The following are equivalent:
    \begin{enumerate}
        \item $M \leqdeg N$,
        \item there exists a complex $Z \in \bchnproj$ and a short exact sequence $ 0 \rightarrow N \rightarrow M \oplus Z \rightarrow Z \rightarrow 0$,
        \item there exists a complex $Z \in \bchnproj$ and a short exact sequence $ 0 \rightarrow Z \rightarrow M \oplus Z \rightarrow N \rightarrow 0$.
    \end{enumerate}
\end{theorem*}
In fact, the complex $Z$ constructed in the proof will be supported on the same interval $[a,b]$ as $M$ and $N$. So this Theorem may be restricted to $\ichnproj{a}{b}$. We remark that neither the Theorem nor its restriction can be derived from \cite{JSZ_ch_cx} although there are some similarities in the proofs connecting algebraic and geometric degenerations due to using analogues of \cite{Rie,Z}. 
We also introduce analogues of the ext and hom order and show that the classical implications still hold.

In the second part, we then define an ind-variety $\bvarPhom$ whose points are chain complexes up to homotopy equivalence for every choice of $g$-vector, i.e. element in the Grothendieck group of $\bhomproj$. This ind-variety is a filtered colimit of chain complex varieties where the colimit structure is given by adding acyclic complexes. Similarly, we define an ind-group acting on the ind-variety whose orbits are the homotopy classes of complexes. We emphasize that this ind-variety is not the topological space described in \cite{JSZ_ch_cx} for right bounded complexes and does not model the same concept. A degeneration in the ind-variety exists if and only if we can add acyclic complexes to the complexes involved to find a degeneration in some specific $\varP$. These geometric degenerations are exactly the algebraic degenerations in the homotopy category.
\begin{theorem*}[\Cref{thm:deg_traingles_homotopy}]
    Let $M,N \in \bvarPhom$. The following are equivalent:
    \begin{enumerate}
        \item $M \leqdeg N$,
        \item there exists a complex $Z \in \bhomproj$ and a triangle $ N \rightarrow M \oplus Z \rightarrow Z \rightarrow N[1]$,
        \item there exists a complex $Z \in \bhomproj$ and a triangle $ Z \rightarrow M \oplus Z \rightarrow N \rightarrow Z[1]$.
    \end{enumerate}
\end{theorem*}

This construction also restricts to chain complexes bounded on an interval $[a,b]$ so we can study the ind-variety $\ivarPhom{a}{b}$. We then show that degenerations in $\ivarPhom{a}{b}$ always come from degenerations in the $\bvarPhom$. Next, we show that in the case of finite global dimension, any module degeneration gives rise to a degeneration in $\bvarPhom$. Finally, we define analogues of the ext and hom order and show that the usual implications hold. 

\bigskip
This paper is structured as follows. In \Cref{sect:Cb}, we discuss varieties and degenerations of chain complexes, in \Cref{sect:Kb} we recall the concept of an ind-variety and introduce the ind-variety associated to $\bhomproj$ and its degenerations.

For an introduction to the representation theory of path algebras and representations we refer to \cite{ASS}. A background of varieties and algebraic groups can be found in \cite{Ha}. We also refer to the survey \cite{Z_survey} on representation varieties. Details on colimits and categorical notions may be referred to in \cite{ML}. For more context of chain complexes and the homotopy category we cite \cite{W}.

We compose arrows as morphisms and work with left modules, identifying modules with representations of quivers in the usual way.

\biblio

\section{Chain complexes of projective modules}
\label{sect:Cb}

Let $\Lambda$ be a finite dimensional algebra. We denote the category of bounded chain complexes whose underlying modules are finite dimensional by $\bchnmod$. For a complex $X \in \bchnmod$, we call the $(X^i)_{i\in \bbZ}$ the underlying modules of $X$. Since $X$ is bounded, only finitely many of the $X^i$ are non-zero. We also say that $(X^i)_{i\in \bbZ}$ is finitely supported.

For integers $a\leq b$, we let $\ichnmod{a}{b}$ be the subcategory of $\bchnmod$ whose complexes are supported on the interval $[a,b]$. We remark that $\bchnmod$ is an abelian category where short exact sequences are those which are exact in $\mod \Lambda$ in every degree. The subcategory $\ichnmod{a}{b}$ naturally inherits this structure.

Moreover, we denote the exact subcategory of $\bchnmod$ of complexes whose underlying modules are projective by $\bchnproj$. Similarly, we define $\ichnproj{a}{b}$ for $a \leq b$. The exact structure is inherited from $\bchnmod$ and given by admissible short exact sequences which are degree-wise split.

Note that any complex in $\bchnproj$ is contained in an $\ichnproj{a}{b}$ where the support of the complex is contained in $[a,b]$. We will primarily study $\chnproj$ where $n \in \bbN$. By shifting degrees, results obtained there also hold for any $\ichnproj{a}{b}$.

\begin{remark}
\label{rmk:complexes_as_modules}
Note that $\chnmod$ can be interpreted as the category of finite dimensional modules over $A_{n+1} \otimes \Lambda$ where $A_{n+1} = k(n \rightarrow (n-1) \rightarrow \ldots \rightarrow 1 \rightarrow 0)/\rad^2$.
\end{remark}

We will later need the following relation between $\chnproj$ and $\chnmod$.

\begin{lemma}
\label{lem:proj_contravariantly_finite}
    The subcategory $\chnproj$ is contravariantly finite in $\chnmod$.
\end{lemma}

\begin{proof}
Let $M = (M^{-n} \xrightarrow{d_M^{-n}} M^{-(n-1)} \rightarrow \ldots \rightarrow M^{-1} \xrightarrow{d^{-1}} M^0) \in \chnmod$. We build a right $\chnproj$ approximation $P$ as described in the diagram below.
\[\begin{tikzcd}
	{P^{-n}} \\
	{E^{-n}} & \ldots & {P^{-2}} \\
	&& {E^{-2}} & {P^{-1}} \\
	&&& {E^{-1}} & {P^0} \\
	{M^{-n}} & \ldots & {M^{-2}} & {M^{-1}} & {M^{0}}
	\arrow["{p^{-n}}", two heads, from=1-1, to=2-1]
	\arrow[from=2-1, to=2-2]
	\arrow[from=2-1, to=5-1]
	\arrow[from=2-2, to=2-3]
	\arrow["{p^{-2}}", two heads, from=2-3, to=3-3]
	\arrow["{e^{-2}}", from=3-3, to=3-4]
	\arrow["\lrcorner"{anchor=center, pos=0.125}, draw=none, from=3-3, to=4-4]
	\arrow["{q^{-2}}"', from=3-3, to=5-3]
	\arrow["{p^{-1}}", two heads, from=3-4, to=4-4]
	\arrow["{e^{-1}}", from=4-4, to=4-5]
	\arrow["{q^{-1}}"', from=4-4, to=5-4]
	\arrow["\lrcorner"{anchor=center, pos=0.125}, draw=none, from=4-4, to=5-5]
	\arrow["{p^0}"', two heads, from=4-5, to=5-5]
	\arrow[from=5-1, to=5-2]
	\arrow[from=5-2, to=5-3]
	\arrow["{\varphi^-2}", from=5-3, to=4-4]
	\arrow["{d_M^{-2}}", from=5-3, to=5-4]
	\arrow["{d_M^1}", from=5-4, to=5-5]
\end{tikzcd}\]
Here, $P^0 \xrightarrow{p^0} M^0$ is a projective cover of $M^0$, $E^{-1}$ is the pullback of $d_M^{-1}$ and $p^0$. Iteratively, for all $1 \leq i \leq n$, $P^{-i} \xrightarrow{p^{-i}} M^{-i}$ is a projective cover of $E^{-i}$. Since $d_M^{-i}d_M^{-i-1} = 0$, we can use the universal property of $E^{-i}$ to find a morphism $\varphi^{-i-1}\colon M^{-i-1} \rightarrow E^{-i}$ such that $e^{-i}\varphi^{-i-1} = 0$ and $q^{-i}\varphi^{-i-1} = d_M^{-i-1}$. We let $E^{-i-1}$ be the pullback of $p^{-i}$ and $\varphi^{-i-1}$. The approximation is given by
\[\begin{tikzcd}
	{P^{-n}} & \ldots & {P^{-2}} & {P^{-1}} & {P^0} \\
	{M^{-n}} & \ldots & {M^{-2}} & {M^{-1}} & {M^{0}}
	\arrow[from=1-1, to=1-2]
	\arrow["{q^{-n} p^{-n}}"', from=1-1, to=2-1]
	\arrow[from=1-2, to=1-3]
	\arrow["{e^{-2}\circ p^{-2}}", from=1-3, to=1-4]
	\arrow["{q^{-2} p^{-2}}"', from=1-3, to=2-3]
	\arrow["{e^{-1} \circ p^{-1}}", from=1-4, to=1-5]
	\arrow["{q^{-1}p^{-1}}"', from=1-4, to=2-4]
	\arrow["{p^0}"', from=1-5, to=2-5]
	\arrow[from=2-1, to=2-2]
	\arrow[from=2-2, to=2-3]
	\arrow["{d_M^{-2}}", from=2-3, to=2-4]
	\arrow["{d_M^{-1}}", from=2-4, to=2-5]
\end{tikzcd}\]
Note that the first line is indeed a complex as for all $0 \leq i \leq n-1$, we have 
\begin{equation*}
    e^{-i} p^{-i} e^{-i-1} p^{-i} = \underbrace{e^{-i} \varphi^{-i-1}}_{= 0} q^{-i-1} p^{-i} = 0.
\end{equation*}

To show that this is a right approximation, let $Q \in \chnproj$ and $f \colon Q \rightarrow M$ a map in $\chnmod$
\[\begin{tikzcd}
	{Q^{-n}} & \ldots & {Q^{-2}} & {Q^{-1}} & {Q^0} \\
	{M^{-n}} & \ldots & {M^{-2}} & {M^{-1}} & {M^{0}.}
	\arrow[from=1-1, to=1-2]
	\arrow["{f^{-n}}"', from=1-1, to=2-1]
	\arrow[from=1-2, to=1-3]
	\arrow["{d_Q^{-2}}", from=1-3, to=1-4]
	\arrow["{f^{-2}}"', from=1-3, to=2-3]
	\arrow["{d_Q^{-1}}", from=1-4, to=1-5]
	\arrow["{f^{-1}}"', from=1-4, to=2-4]
	\arrow["{f^0}"', from=1-5, to=2-5]
	\arrow[from=2-1, to=2-2]
	\arrow[from=2-2, to=2-3]
	\arrow["{d_M^{-2}}", from=2-3, to=2-4]
	\arrow["{d_M^{-1}}", from=2-4, to=2-5]
\end{tikzcd}\]
We construct a lifting $h \colon Q \rightarrow P$ inductively. We discuss the first two steps. Since $p^0$ is surjective and $Q^0$ projective, $f^0$ factors through $p^0$, i.e. there is a morphism $h^0 \colon Q^0 \rightarrow P^0$ such that $p^0h^0 = f^0$. Since $f$ is a morphism of complexes, we see that $d_M^{-1}f^{-1} = f^0d_Q^{-1} = p^0h^0d_Q^{-1}$. So by the universal property of $E^{-1}$, there exists a morphism $g^{-1}\colon Q^{-1} \rightarrow E^{-1}$ such that $q^{-1}g^{-1} = f^{-1}$ and $e^{-1}g^{-1} = h^0d_Q{-1}$. As $p^{-1}$ is an epimorphism and $Q^{-1}$ projective, we can find a morphism $h^{-1} \colon Q^{-1} \rightarrow E^{-1}$ such that $p^{-1}h^{-1} = g^{-1}$. Then we have
\begin{align*}
    q^{-1}p^{-1}h^{-1}d_Q^{-2} = g^{-1}q^{-1}d_Q^{-2} = f^{-1}d_Q^{-2} = d_M^{-2}f^{-2} = d_M^{-2}f^{-2} = q^{-1}\varphi^{-2}f^{-2}.
\end{align*}
Since $q^{-1}$ is surjective, we have $p^{-1}h^{-1}d_Q^{-2} = \varphi^{-2}f^{-2}$. So we can use the universal property of $E^{-2}$ to find a morphism $g^{-2} \colon Q^{-2} \rightarrow E^{-2}$ such that $q^{-2}g^{-2} = f^{-2}$ and $e^{-2}g^{-2} = h^{-1}d_Q^{-2}$. Iterating the construction of $h$ and $g$ yields a morphism $h \colon Q \rightarrow P$ showing that $f$ factors through the approximation $P$. Note that $h$ is indeed a morphism of complexes as by design $e^{-i}p^{-i}h^{-i} = e^{-i}g^{-i} = h^{-i+1}d_Q^{-i}$.
\end{proof}

\subsection{Varieties of chain complexes}
\label{subsect:var_of_Cb}

We fix $\PP = (P^i)_{i \in \bbZ}$ a finitely supported sequence of projective modules. Let $[a,b]$ be the smallest interval such that $\PP$ is supported in it. A chain complex $M \in \bchnproj$ with underlying modules $\PP$ is of the form
\begin{equation*}
    M = (\cdots \rightarrow P^{i-1} \xrightarrow{d_M^{i-1}} P^{i} \xrightarrow{d_M^{i}} P^{i+1} \rightarrow \cdots).
\end{equation*}
In particular, provided we know $\PP$, a complex is completely given by its differentials. We write $M = (d_M^{i})_{i \in \bbZ}$ for short. This gives rise to the variety of chain complexes with underlying modules $\PP$
\begin{equation*}
    \varP \subset \prod_{a \leq i < b} \Hom_\Lambda(P^{i}, P^{i+1}).
\end{equation*}
The group $\GP = \prod_{a \leq i \leq b} \Aut(P^{i})$ acts on it by conjugation:
\begin{equation*}
    g.M := (g_{P^{i+1}} d_M^{i} g_{P^{i}}^{-1}).
\end{equation*}
Note that the orbit of $M$ for this action, $\sO_M$, is the isomorphism class of $M$ in $\bchnproj$. In \cite{JSZ_ch_cx}, where $\varP$ and $\GP$ were first introduced, they are denoted by $comproj^{\underline{d}}$ and $G$ where $\underline{d}$ is a bounded vector encoding $\PP$.

\begin{example}
\label{ex:A_1_chain_cx}
We consider the algebra $\Lambda = kA_1 = k$. So the objects in $\mod \Lambda$ are the finite dimensional $k$-vector spaces. The only indecomposable projective module up to isomorphism is $P = k$. A choice of underlying modules $\PP$ is completely determined by the dimension of each $P^i$. We define $d = d(\PP) = (\dim (P^i))_{i \in \bbZ}$. Without loss of generality, we assume that $\PP$ is supported on the interval $[-n,0]$. Then a point $X \in \varP$ is given by $f \in \prod_{-n \leq i \leq -1} \Mat_{d_{i+1}\times d_i}(k)$ such that $f^i \circ f^{i-1} = 0$ for all $i$. So $\varP = \rep(B, d(\PP))$ where $B = k(A_{n+1})/\rad^2$. Moreover, $\GP = \prod_{-n\leq i \leq 0} \GL_{d_i} = \GL_{d(\PP)}$ and the group action is the same. 
\end{example}

\begin{remark}
\label{rmk:rel_chain_cx_to_rep_var}
Note that in general, varieties of chain complexes are not obtained this way. While chain complexes can be interpreted as modules over $\Lambda \otimes B$, the underlying modules cannot generally be uniquely depicted by a dimension vector. We will se examples of this in \Cref{sect:examples}
\end{remark}

\begin{definition}
\label{def:degeneration}
    Let $M,N \in \varP$. We say that $M$ degenerates to $N$, or that $N$ is a degeneration of $M$, written $M \leqdeg N$, if there is an inclusion $N \in \overline{\sO_M}$.
\end{definition}

\begin{example}
\label{ex:trivial_deg}
    Let $\Lambda$ be any finite dimensional algebra. We consider projective modules $P$ and $Q$ with a morphism $0 \neq f \colon P \rightarrow Q$. Then $(P \xrightarrow{f} Q) \leqdeg (P \xrightarrow{0} Q)$. This can be directly seen by letting $\lambda$ go to zero in the family $(P \xrightarrow{\lambda f} Q)_{\lambda \in k}$.
\end{example}

\begin{example}
\label{ex:A_1_chain_cx_deg}
    We continue the study of $\Lambda = kA_1$ in \Cref{ex:A_1_chain_cx}. Let $m \geq 1$ and $d(\PP) = (d_{-n}, \ldots, d_0)$ where $d_{-n} = d_0 = m$ and $d_i = 2m$ else. A complex $C$ over $\PP$ is isomorphic to a direct sum of indecomposable complexes which are, up to isomorphism, the stalk complexes obtained as a shift of $k$ and minimal non-zero acyclic complexes $(k \xrightarrow{\id}k)$. The different complexes are, up to isomorphism, completely determined by their cohomology dimensions. Let $a(C) = (a_i)_{-n \leq i \leq 0}$ with $a_i = \dim \mathrm{H}^i(C)$. Incidentally, $a_i$ also counts the number of stalk complexes concentrated in degree $i$. 
    
    So the orbits of $\varP$ are parametrised by vectors $a = (a_i) \in \bbN^{n+1}$ satisfying
    \begin{enumerate}
        \item \label{item:dim} The vector respects the dimension, i.e. $a \leq d(\PP)$.
        \item \label{item:complement} The complement of $a$, $d(P) -a$, supports an acyclic complex, i.e. $d(P) - a = (b_{-n}, b_{-n}+ b_{-n+1}, \ldots, b_{-2} + b_{-1}, b_{-1})$ where each $b_i \geq 0$. This can be reformulated to the two following conditions which only depend on $a$ and $d(\PP)$:
        \begin{enumerate}
            \item \label{item:sum_zero} We can decompose $d(\PP)- a$ into the $b_i$s: $\sum_{-n \leq i \leq 0} (-1)^i a_i = 0$.
            \item \label{item:not_zero} The $b_i$ are not negative: $b_i = (-1)^{n-1} (m-a_{-n}) + \sum_{j=i}^{n-1} (-1)^{j-i} (2m-a_{-i}) \geq 0$.
        \end{enumerate}
    \end{enumerate}
    
    We claim that the cover relations of the degeneration order are generated by $(k \xrightarrow{\id}k) \leqdeg (k \xrightarrow{0}k)$. This translates to $a \lessdot c$ if and only if there exists a unique $i$ such that $a_i = c_i-1$, $a_{i+1} = c_{i+1}-1$ and $a_j = c_j$ else. 
    
    To prove this statement, we first note that by \Cref{ex:trivial_deg}, $a \lessdot c$ is indeed a degeneration. We need to show that it is covering and that any non-trivial degeneration factors through degenerations of this type. Let us assume that $M <_\text{deg} N$. Since we can view the variety as a representation variety, we can use that the rank function is lower-semicontinuous. So for any differential $\rk d_M^i \geq \rk d_N^i$. In our setting, if for every $i$, $\rk d_M^i = \rk d_N^i$, then $M \cong N$ but we assumed this is not the case. So there exists an $i$ with $\rk d_M^i > \rk d_N^i$. So an identity degenerates into a zero and we factor through some $a \lessdot c$. Then we can iterate. Since the $a \lessdot c$ change exactly one rank by one, they are covering.
    
    For a concrete example, we consider $n = 1$. In this case, $a = (a_{-1}, a_0)$ where $a_{-1} = a_0 \leq m$. Then the poset of degenerations is given by 
    \[\begin{tikzcd}[sep=small]
    	{(m,m)} && {k^m \xrightarrow{0}k^m} \\
    	{(m-1,m-1)} && {k^m \xrightarrow{\diag(0,\ldots,0,1)}k^m} \\
    	\vdots & \longleftrightarrow & \vdots \\
    	{(1,1)} && {k^m \xrightarrow{\diag(0,1,\ldots,1)}k^m} \\
    	{(0,0)} && {k^m \xrightarrow{\id}k^m.}
    	\arrow[no head, from=1-1, to=2-1]
    	\arrow[no head, from=1-3, to=2-3]
    	\arrow[no head, from=2-1, to=3-1]
    	\arrow[no head, from=2-3, to=3-3]
    	\arrow[no head, from=3-1, to=4-1]
    	\arrow[no head, from=3-3, to=4-3]
    	\arrow[no head, from=4-1, to=5-1]
    	\arrow[no head, from=4-3, to=5-3]
    \end{tikzcd}\]
    
    For $n = 2$, we obtain vectors of the form $a = (a_{-2}, a_{-1}, a_0)$ where $a_{-1} = a_{-2} + a_0$ and $a_0$ and $a_{-2}$ are less or equal than $m$. The poset of degenerations is thus a $m\times m$-grid. For $m = 2$, we obtain:
    \[\begin{tikzcd}[sep=small]
    	&& {(2,4,2)} && \\
    	& {(2,3,1)} && {(1,3,2)} \\
    	{(2,2,0)} && {(1,2,1)} && {(0,2,2)} \\
    	& {(1,1,0)} && {(0,1,1)} \\
    	&& {(0,0,0)}
    	\arrow[no head, from=1-3, to=2-2]
    	\arrow[no head, from=1-3, to=2-4]
    	\arrow[no head, from=2-2, to=3-1]
    	\arrow[no head, from=2-2, to=3-3]
    	\arrow[no head, from=2-4, to=3-3]
    	\arrow[no head, from=2-4, to=3-5]
    	\arrow[no head, from=3-1, to=4-2]
    	\arrow[no head, from=3-3, to=4-2]
    	\arrow[no head, from=3-3, to=4-4]
    	\arrow[no head, from=3-5, to=4-4]
    	\arrow[no head, from=4-2, to=5-3]
    	\arrow[no head, from=4-4, to=5-3]
    \end{tikzcd}\]
    
    For $n \geq 3$, we no longer have a $m^{n-1}$-grid. In particular, there are more minimal elements than $a = 0$. We consider $n=3$ and $m = 1$. The poset of degenerations is given by the following:
    \[\begin{tikzcd}[sep=small]
    	&& {(1,2,2,1)} & \\
    	& {(1,1,1,1)} & {(1,2,1,0)} & {(0,1,2,1)} \\
    	{(1,0,0,1)} & {(1,1,0,0)} & {(0,0,1,1)} & {(0,1,1,0)} \\
    	&& {(0,0,0,0)}
    	\arrow[no head, from=1-3, to=2-2]
    	\arrow[no head, from=1-3, to=2-3]
    	\arrow[no head, from=1-3, to=2-4]
    	\arrow[no head, from=2-2, to=3-1]
    	\arrow[no head, from=2-2, to=3-2]
    	\arrow[no head, from=2-2, to=3-3]
    	\arrow[no head, from=2-3, to=3-2]
    	\arrow[no head, from=2-3, to=3-4]
    	\arrow[no head, from=2-4, to=3-3]
    	\arrow[no head, from=2-4, to=3-4]
    	\arrow[no head, from=3-2, to=4-3]
    	\arrow[no head, from=3-3, to=4-3]
    	\arrow[no head, from=3-4, to=4-3]
    \end{tikzcd}\]
    Note that there are two minimal elements: $(0,0,0,0)$ and $(1,0,0,1)$. These represent the two irreducible components of the variety $\varP$. The vector $(0,0,0,0)$ corresponds to the open orbit of the acyclic complex $(k \xrightarrow{\pmat{1\\0}} k^2 \xrightarrow{\pmat{0 & 1 \\ 0 & 0}} k^2 \xrightarrow{\pmat{1 & 0}} k)$. The closure of this orbit is the irreducible component
    \begin{equation*}
        I_1 = \left\{ X = (k \xrightarrow{M_1} k^2 \xrightarrow{M_2} k^2 \xrightarrow{M_3} k) \mid X \in \varP, \forall i\in \{1,2,3\}, \rk(M_i) \leq 1 \right\}.
    \end{equation*}
    
    Similarly, the vector $(1,0,0,1)$ corresponds to the open orbit of the chain complex $(k \xrightarrow{0} k^2 \xrightarrow{\id} k^2 \xrightarrow{0} k)$. Its closure is the irreducible component 
    \begin{equation*}
        I_2 = \left\{ X = (k \xrightarrow{0} k^2 \xrightarrow{M} k^2 \xrightarrow{0} k) \mid X \in \varP \right\}.
    \end{equation*}
    
    In general, for $n=3$, the number of orbits are enumerated by the house numbers\footnote{\url{https://oeis.org/A051662}} starting at $9$ for $m=1$.
\end{example}

\subsection{Degenerations and short exact sequences}
\label{subsect:deg_in_Cb}

Again, let $\PP = (P^i)_{i \in \bbZ}$ a finitely supported sequence of projective modules. We set $[a,b]$ to be the smallest interval such that $\PP$ is supported in it. 

\begin{theorem}
\label{thm:deg_and_ses_in_chnproj}
    Let $M,N \in \varP$. The following are equivalent:
    \begin{enumerate}
        \item \label{item:deg} $M \leqdeg N$,
        \item \label{item:ses_1} there exists a complex $Z \in \ichnproj{a}{b}$ and a short exact sequence $ 0 \rightarrow N \rightarrow M \oplus Z \rightarrow Z \rightarrow 0$,
        \item \label{item:ses_2} there exists a complex $Z \in \ichnproj{a}{b}$ and a short exact sequence $ 0 \rightarrow Z \rightarrow M \oplus Z \rightarrow N \rightarrow 0$.
    \end{enumerate}
\end{theorem}

We will show the above Theorem in multiple steps. Without loss of generality, we assume that $[a,b] = [-n,0]$ for some $n \in \bbN$. To show that a short exact sequence as in (\ref{item:ses_2}) gives rise to a degeneration, we use an analogue of the proof from \cite{Rie} in \Cref{lem:analogue_Riedtmann}. To show that degenerations give rise to short exact sequences as in (\ref{item:ses_1}), we adapt the strategy of \cite{Z} in \Cref{lem:analogue_Zwara} using a geometric theorem on degenerations in \cite{GO} stated in \Cref{lem:geometry}. This will only give rise to a short exact sequence where $Z \in \bchnmod$, i.e. the underlying modules of $Z$ not necessarily projective. However, we can replace this $Z$ by one in $\bchnproj$ using the contravariant finiteness of $\bchnproj$ in $\bchnmod$, see \Cref{lem:ses_chnmod_to_chnproj}. The remaining implications are then obtained by dualising the first two steps. Another analogue of Riedtmann and Zwara was done in  \cite{JSZ_ch_cx}. Our Riedtmann analogue is a special case of theirs. The Zwara analogues differ and the resulting theorems are not equivalent and cannot be derived from each other.

We begin with the analogue of \cite{Rie}.
\begin{lemma}
\label{lem:analogue_Riedtmann}
    Let $0 \rightarrow Z \xrightarrow{\pmat{\xi \\ \sigma}} M \oplus Z \xrightarrow{g} N \rightarrow 0$ be a short exact sequence in $\chnproj$ with $M,N \in \varP$. Then $M \leqdeg N$.
\end{lemma}

\begin{proof}
    We adapt the strategy of \cite{Rie} who treats the $\mod \Lambda$ case. Let $-n \leq i \leq 0$.
    Since the sequence is short exact in $\chnproj$, it is line-wise a split short exact sequence. So there is a section $\varphi^i$ of $g^i$. Note that the $\varphi^i$ do not necessarily form a morphism of chain complexes. However, we have $Z^i \oplus M^i \underset{\text{v.sp.}}{\cong} \im(f^i) \oplus \im(\varphi^i)$. We denote $C^i = \im(\varphi^i)$.
    
    The module structure on this vector space decomposition is given, for any $a \in \Lambda$ as the action
    \[
    \NiceMatrixOptions{code-for-first-row=\scriptstyle,code-for-first-col=\scriptstyle}
    [a]_{Z^i \oplus M^i} = \begin{pNiceMatrix}[first-col,first-row]
        & \im(f^i) & C^i\\
        \im(f^i) & [a]_{Z^i} & 0 \\
        C^i & 0 & [a]_{N^i} \\
    \end{pNiceMatrix}.
    \]
    
    The differential with respect to this decomposition is 
    \[
    \NiceMatrixOptions{code-for-first-row=\scriptstyle,code-for-first-col=\scriptstyle}
    d^i = \begin{pNiceMatrix}[first-col,first-row]
        & \im(f^i) & C^i\\
        \im(f^{i+1}) & d_M^i & * \\
        C^{i+1} & 0 & d_N^i \\
    \end{pNiceMatrix}.
    \]
    
    Let $t \in k$. We define $f_t := \pmat{\xi + t \cdot \id_{Z}\\ \sigma}$. Note that $f_0 = f$. Being a line-wise monomoprhism is an open condition on $\chnproj$. So there is a non-empty open set $\sU_{\text{open}} \subset k$ such that for every $t \in \sU_{\text{open}}$, $f^i_t$ is a monomorphims. It is non-empty since it contains $0$. 
    
    For every $t \in \sU_{\text{open}}$, $\dim(\im(f_t^i)) = \dim(\im(f^i))$. So $C^i$ is a complement of $\im(f_t^i)$ if and only if $C^i \cap \im(f_t^i) = 0$. This is an open condition. So there exists a non-empty open subset $\sU_\text{compl} \subset \sU_\text{open}$ such that for every $t \in \sU_{\text{compl}}$, $C^i$ is a complement of $\im(f^i_t)$. It is non-empty since it contains $0$.
    
    For every $t \in \sU_{\text{compl}}$, since $Z^i \oplus M^i \underset{\text{v.sp.}}{\cong} \im(f_t^i) \oplus \im(\varphi^i)$, we can again describe the module and differential structure given this choice of basis. It is given by
    \[
    \NiceMatrixOptions{code-for-first-row=\scriptstyle,code-for-first-col=\scriptstyle}
    [a]_{Z^i \oplus M^i} = \begin{pNiceMatrix}[first-col,first-row]
        & \im(f_t^i) & C^i\\
        \im(f_t^i) & [a]_{Z^i} & * \\
        C^i & 0 & [a]_t^i \\
    \end{pNiceMatrix} \qquad \text{and} \qquad
    \NiceMatrixOptions{code-for-first-row=\scriptstyle,code-for-first-col=\scriptstyle}
    d^i = \begin{pNiceMatrix}[first-col,first-row]
        & \im(f^i) & C^i\\
        \im(f^{i+1}) & d_Z^i & * \\
        C{i+1} & 0 & d^i_t \\
    \end{pNiceMatrix}.
    \]
    
    Let us consider $\xi^i + t \cdot \id_{Z^i}$. It is an endomorphism of $Z^i$. In particular, if it is injective it is also bijective. It is not injective if there exists a non-zero $z \in Z^i$ such that $(\xi^i + t \cdot \id_{Z^i})(z) = 0$. This is the case if and only if $t$ is an eigenvalue of $\xi^i$. Since $\xi^i$ has only finitely many eigenvalues, $\xi^i + t \cdot \id_{Z^i}$ is an isomorphism for nearly all $t$. We let $\sU = \sU_\text{compl} \backslash \cup_i \{\text{eigenvalues of }\xi^i \}$. This is a non-empty open subset of $k$. Note that $\sU \cup \{0\}$ is also open.
    
    For every $t \in \sU$, we consider the short exact sequence (a priori in $\chnmod$)
    \begin{equation*}
        0 \rightarrow Z \xrightarrow{f_t} M \oplus Z \rightarrow \coker(f_t) \rightarrow 0.
    \end{equation*}
    Since $(\xi^i + t \cdot \id_{Z^i})(z)$ is an isomorphism, $\coker(f_t) \cong M$. This implies that $[a]_t^i = [a]_{M^i}$ and $d^i_t = d_M^i$. So we have found a family of chain complex structures $C_t = ([a]_t, d_t)$ parametrised by an open subset $\sU \cup \{0\}$ such that $C_0 = N$ and for every $t \neq 0$, $C_t \cong M$. Thus $M$ degenerates to $N$.
\end{proof}

To construct the analogue of \cite{Z}, we first decode degenerations into the existence of certain points which we will be able to interpret in a representation theoretic way. The following Lemma is an application of \cite[Theorem 1.2]{GO} to our setting.

\begin{lemma}%
\label{lem:geometry}%
Let $M,N \in \varP$. Then $M \leqdeg N$ if and only if there exists a discrete valuation $k$-algebra $R$ with maximal ideal $\frm$ and residue field $R/\frm = k$, whose quotient field $K$ is finitely generated over $k$ of transcendence degree one, and a complex $Y \in \varP(R)$ such that
\begin{itemize}
    \item $\tau^*(Y) = g . (\tau\eta)^*(M)$, for some $g \in \GP$,
    \item $\pi^*(Y) = N$,
\end{itemize}
where $\eta$, $\tau$ and $\pi$ are the canonical homomorphisms
\[\begin{tikzcd}
	k & R & K \\
	& {R/\frm = k.}
	\arrow["\eta", from=1-1, to=1-2]
	\arrow["\tau", from=1-2, to=1-3]
	\arrow["\pi", from=1-2, to=2-2]
\end{tikzcd}\]
Algebraically, this means that
\begin{itemize}
    \item $Y \otimes_R K = g \cdot (M \otimes_k K)$, and
    \item $Y \otimes R/\frm = N$.
\end{itemize}
\end{lemma}

Note that $\varP(R)$ consists of complexes with underlying modules $\PP\otimes_k R = (P^{-n}\otimes_k R, \ldots, P^{-1}\otimes_k R, P^0\otimes_k R)$. In the setting of the above Lemma, certain properties hold that we will need to reference later. These all appear in \cite[p. 210 f.]{Z}.

\begin{lemma}
\label{lem:dvr_properties}
    Let $R$ be a discrete valuation $k$-algebra $R$ with maximal ideal $\frm$ and residue field $R/\frm = k$, whose quotient field $K$ is finitely generated over $k$ of transcendence degree one. The following hold.
    \begin{enumerate}
        \item There exists an element $f \in \frm \backslash \frm^2$ such that for all $i \geq 1$, $\frm^i = (f^i)$.
        \item For every $i \geq 1$, there is a short exact sequence in $\mod R$ of the form
        \begin{equation*}
            0 \rightarrow R/\frm \xrightarrow{\beta_i} R/\frm^{i+1} \xrightarrow{\gamma_i} R/\frm^i \rightarrow 0
        \end{equation*}
        where $\beta_i(r+ \frm) = f^i \cdot r + \frm^{i+1}$ and $\gamma_i(r + \frm^{i+1}) = r + \frm^i$.
        \item Let $\sB$ be a $k$-basis of $R$. Let $\sV \subset \sB$ finite. Then there exists a natural number $s$ and $\sW \subset \sB$ finite such that as a vector space $\frm^s \oplus \langle \sV \rangle \oplus \langle \sW \rangle = R$.
    \end{enumerate}
\end{lemma}

Using these two lemmas, we may now adapt the proof of \cite{Z}.
\begin{lemma}
\label{lem:analogue_Zwara}
    Let $M,N \in \varP$ such that $M \leqdeg N$. Then there exists a complex $Z \in \chnmod$ and a short exact sequence $ 0 \rightarrow N \rightarrow M \oplus Z \rightarrow Z \rightarrow 0$. 
\end{lemma}

\begin{proof}
    We adapt the strategy of \cite{Z} who treats the $\mod \Lambda$ case.
    
    By \Cref{lem:geometry}, using the notation introduced there, we consider the complex $Y \in \varP(R)$ of the form
    \begin{equation*}
        M = (P^{-n} \otimes_k R \xrightarrow{d_Y^{-n}} P^{-n+1} \otimes_k R \rightarrow \ldots \rightarrow P^{-1} \otimes_k R \xrightarrow{d_Y^{-1}} P^0 \otimes_k R).
    \end{equation*}
    Note that $P^i \otimes_k R$ is an $\Lambda$-$R$-bimodule. As an $R$-module, it is free of rank $\dim(P)$. Moreover, using \Cref{rmk:complexes_as_modules}, we can see that $Y$ is free of finite rank over $R$ as an $(A_n \otimes \Lambda)$-$R$-module. So we have an exact functor $\sF = {_\Lambda}Y \otimes_R (-) \colon \mod R \rightarrow \chnmod$. 
    
    For $i \geq 1$, let $N_i = \sF(R/\frm^i) = Y/Y\frm^i$. By \Cref{lem:geometry}, ${_\Lambda}N_1 = {_\Lambda}(Y \otimes_R R/\frm) = {_\Lambda}N$.
    
    Since $\sF$ is exact, the short exact sequence from \Cref{lem:dvr_properties} (2) yields a short exact sequence 
    \begin{equation*}
        0 \rightarrow N_1 \rightarrow N_{i+1} \rightarrow N_i \rightarrow 0.
    \end{equation*}
    
    Since $N_1 = N$, it remains to show that for some $i$, $N_{i+1} \cong N_i \oplus M$. 
    
    Recall that for some $g \in \GP(K)$, $Y \otimes_R K = g \cdot (M \otimes_k K)$. Let $f \in R$ as in \Cref{lem:dvr_properties} (1). Since $g$ is given by matrices with entries in $K$, we can identify a smallest power of $f$, $f^j$, such that every entry of $f^j \cdot g$ lies in $R$. The inverse of $f^j \cdot g$ is given by $f^{-j} \cdot g^{-1}$. For all $i$, 
    \begin{equation*}
        f^{-j} \cdot g_{i-1}^{-1} d_M^i f^{j} \cdot g_{i} = f^{-j} \cdot f^{j} \cdot g_{i-1}^{-1} d_M^i g_{i} = g_{i-1}^{-1} d_M^i g_{i}.
    \end{equation*}
    So the action of $g$ and $f^{j} \cdot g$ are the same and we can always assume that all entries of $g$ lie in $R$. 
    
    This means that we can restrict the isomorphism $g \colon M \otimes_k K \rightarrow Y \otimes_R K$ to a morphism $\varphi = g|_{M \otimes_k R} \colon M \otimes_k R \rightarrow Y$ which is line-wise a monomorphism of $\Lambda$-$R$-bimodules. It is not necessarily a monomorphism of $\chnproj$. We let $X = \im \varphi \in \chnmod(R)$. Here, the image is simply the line-wise image. The complex structure is naturally inherited from $Y$.
    
    We want to show that there is a natural number $t$ with $Y\frm^t \subset X$ and that $Y/X \in \chnmod$, i.e. the modules appearing in the complex are finite dimensional as $\Lambda$-modules. Note that this complex is not necessarily in $\chnproj$, i.e. the modules might not be projective.
    
    Line-wise, $\varphi$ is a monomorphism $\varphi^i \in \End(P^i \otimes_k R)$. Recall that $P^i \otimes_k R$ is a free $R$-module of finite rank. Since $R$ is a discrete valuation ring and thus a principal ideal domain, $P^i \otimes_k R / \varphi^i(P^i \otimes_k R)$ is isomorphic to some $R/\frm^{t_1} \oplus \ldots \oplus R/\frm^{t_{\dim(P)}}$. Let $t$ be the maximum of all $t_j$ for all $i$. Then for all $i$, $P^i \otimes_k R / \varphi^i(P^i \otimes_k R)$ is annihilated by $\frm^t$. In particular, $Y\frm^t \subset X$ and thus $Y/X \subset Y/Y\frm^t \in \chnmod$.
    
    Next, we want to show that there exists a natural number $s$ such that $X \frm^s$ is a direct summand of $Y$ in $\chnmod$. Note that we do not require a complement in $\chnproj$.
    
    Let $\sB$ be a $k$-basis of $R$. For $b \in \sB$, we define $M_b = \varphi (M \otimes_k \langle b \rangle) \subset X$. Seen as an element in $\chnmod$, it is isomorphic to $M$. Moreover, $M \otimes_k R \cong \bigoplus_{b \in \sB} M_b$.
    
    We take the projective cover $P \xrightarrow{\xi} Y/X$ in $\chnmod$ and lift this map to $Y$.
    \begin{equation*}
        \begin{tikzcd}
        	& Y \\
        	P & {Y/X.}
        	\arrow[two heads, from=1-2, to=2-2]
        	\arrow["{\xi'}", dashed, from=2-1, to=1-2]
        	\arrow["\xi"', from=2-1, to=2-2]
        \end{tikzcd}
    \end{equation*}
    Let $Z_0 = \im \xi'$. Note that by construction, $Z_0 \in \chnmod$. Moreover, $Z_0 + X = Y$. This is not necessarily a direct sum. The intersection $Z_0 \cap X = Z_0 \cap (\bigoplus_{b \in \sB} M_b)$. Since $Z_0$ consist of finitely many finite dimensional $\Lambda$-modules, this intersection is finite, i.e. there exists a finite subset $\sV \subset \sB$ such that $Z_0 \cap (\bigoplus_{b \in \sB} M_b) = Z_0 \cap (\bigoplus_{b \in \sV} M_b)$. We let $Z = Z_0 + \bigoplus_{b \in \sV} M_b$. Then $Z \in \chnmod$ and $Z \cap X = \bigoplus_{b \in \sV} M_b$.
    
    Now let $C = \bigoplus_{b \in \sB \backslash \sV} M_b$. Then $Y = Z \oplus C$.
    
    We use \Cref{lem:dvr_properties} (3) on $\sV$ to find $R = \frm^s \oplus \langle \sV \rangle \oplus \langle \sW \rangle$.
    So in $\chnmod$, 
    \begin{equation*}
        X = \underbrace{\varphi(M \otimes_k \frm^s)}_{= X \frm^s} \oplus \underbrace{\bigoplus_{b \in \sV} M_b}_{= Z \cap X} \oplus \underbrace{\bigoplus_{b \in \sW}M_b}_{=: W}.
    \end{equation*}
    Let $C' = X \frm^s \oplus W$. We show that $Y = C' \oplus Z = X \frm^s \oplus W \oplus Z$.\\
    First, we use $Z \cap X \subset Z$ and $C' + (Z \cap X) = X$ to conclude that $ C' + Z = (C' + Z) + (Z \cap X) = Z + X = Y$.
    Secondly, we use that $C' \subset X$ to conclude that $C' \cap Z = C' \cap Z \cap X = 0$.
    
    So $X \frm^s$ is indeed a direct summand of $Y$.
    
    Now let us put these two parts together. Let $Z_1 = C'$ and $Z_2 = X/Y \frm^t$. Note that for any $j$, by multiplication with $f^j$, $Z_2 \cong X \frm^i/Y \frm^{t+i}$. Moreover, we can use that $Y = X \frm^s \oplus Z_1$ and $Y\frm^{s+t} \subset X\frm^{s}$ to show that
    \begin{equation*}
        Y/Y\frm^{s+t} = (X \frm^s \oplus Z_1)/Y\frm^{s+t} = X \frm^s/Y\frm^{s+t} \oplus Z_1  = Z_1 \oplus Z_2.
    \end{equation*}
    Finally, we remark that in $\chnmod$, $X\frm^{j+1} \oplus M \cong X\frm^{j}$.
    
    For any $i$, we have the following chain of inclusions in $\chnmod$: 
    \begin{equation*}
        Y \frm^{s+t+i} \subset X \frm^{s+i}\subset X\frm^s\subset Y.
    \end{equation*}
    Here, the last two inclusion have direct complements, $M^i$ and $Z_1$. So 
    \begin{align*}
        Y/Y \frm^{s+t+i} &\cong (X \frm^{s+i} \oplus M^i \oplus Z_1)/Y \frm^{s+t+i}\\
        &\cong X \frm^{s+i}/Y \frm^{s+t+i}  \oplus M^i \oplus Z_1 \\
        &\cong Z_2 \oplus M^i \oplus Z_1 \cong Y/Y\frm^{s+t} \oplus M^{i}.
    \end{align*}
    In particular, for $i = 1$, we find that $N_{s+t+1} \cong N_{s+t} \oplus M$.
\end{proof}

It remains to replace the chain complex $Z$ by one whose underlying modules are projective.
\begin{lemma}
\label{lem:ses_chnmod_to_chnproj}
    Let $M$, $N$ be two complexes in $\chnproj$ such that there exists a short exact sequence $ 0 \rightarrow N \rightarrow M \oplus Z \rightarrow Z \rightarrow 0$ where $Z \in \chnmod$. Then there exists a complex $P \in \chnproj$ and a short exact sequence $ 0 \rightarrow N \rightarrow M \oplus P \rightarrow P \rightarrow 0$.
\end{lemma}

\begin{proof}
We use that $\bchnproj$ is contravariantly finite in $\bchnmod$, see \Cref{lem:proj_contravariantly_finite}. Let $P \rightarrow Z$ be a right minimal $\chnproj$-approximation. We consider the short exact sequence obtained by the pullback of this approximation.
\[\begin{tikzcd}
	0 & N & E & P & 0 \\
	0 & N & {M \oplus Z} & Z & 0
	\arrow[from=1-1, to=1-2]
	\arrow[from=1-2, to=1-3]
	\arrow[from=1-2, to=2-2]
	\arrow[from=1-3, to=1-4]
	\arrow["\rho", from=1-3, to=2-3]
	\arrow["\lrcorner"{anchor=center, pos=0.125}, draw=none, from=1-3, to=2-4]
	\arrow[from=1-4, to=1-5]
	\arrow["\delta", from=1-4, to=2-4]
	\arrow[from=2-1, to=2-2]
	\arrow[from=2-2, to=2-3]
	\arrow[from=2-3, to=2-4]
	\arrow[from=2-4, to=2-5]
\end{tikzcd}\]
We will show that the pullback $E$ is isomorphic to $M \oplus P$. 

We first show that $\rho$ is a $\chnproj$-approximation. Let $Q \rightarrow M \oplus Z$ be a map from $Q \in \chnproj$. This map factors through $\rho$ in two steps, described in the diagram below.
\[\begin{tikzcd}
	Q & E & P \\
	& {M\oplus Z} & Z
	\arrow["{\exists \psi}", dashed, from=1-1, to=1-2]
	\arrow["{\exists \varphi}", curve={height=-24pt}, dashed, from=1-1, to=1-3]
	\arrow[from=1-1, to=2-2]
	\arrow[from=1-2, to=1-3]
	\arrow["\lrcorner"{anchor=center, pos=0.125}, draw=none, from=1-2, to=2-3]
	\arrow["\rho", from=1-2, to=2-2]
	\arrow["\delta", from=1-3, to=2-3]
	\arrow[from=2-2, to=2-3]
\end{tikzcd}\]
First, the compostion $Q \rightarrow M \oplus Z \rightarrow Z$ factors via $\varphi$ through $\delta$ since $\delta$ is a $\chnproj$-approximation. Using the universal property of the pullback $E$, we can then find a map $\psi$ showing that $Q \rightarrow M \oplus Z$ factors through $\rho$.

Next, we consider the map $M \oplus P \xrightarrow{\pmat{\id_M & \\ & \delta}} M \oplus Z$. This is a right minimal $\chnproj$-approximation, inheriting both traits from $\id_M$ and $\delta$ directly.

So $E \cong M \oplus P \oplus C$ where $C \in \chnproj$. In terms of the Grothendieck group, we can thus see that $[E] = [M] + [P] + [C]$. From the original short exact sequence involving $M$, we know that $[M] = [N]$, from the short exact sequence involving $E$, that $[E] = [N] + [P]$. Combining all three yields $[C] = 0$ and thus $C = 0$. So $E \cong M \oplus P$.
\end{proof}

Now we have all ingredients to prove our first main result.

\begin{proof}[Proof of \Cref{thm:deg_and_ses_in_chnproj}]
The implication (\ref{item:ses_2}) $\Rightarrow$ (\ref{item:deg}) was shown in \Cref{lem:analogue_Riedtmann}. For (\ref{item:ses_1}) $\Rightarrow$ (\ref{item:deg}), we can consider the dual short exact sequence in $\Lambda^\text{op}$ where we can apply \Cref{lem:analogue_Riedtmann} and obtain a degeneration for $\Lambda^\text{op}$. Geometrically, the duality translatres to an isomorphism of varieties thus yielding a degeneration for $\Lambda$. 

The implication (\ref{item:deg}) $\Rightarrow$ (\ref{item:ses_1}) follows from combining \Cref{lem:analogue_Zwara} with \Cref{lem:ses_chnmod_to_chnproj}. To show (\ref{item:deg}) $\Rightarrow$ (\ref{item:ses_2}), we use that the equivalence of $\chninj$ and $\chnproj$ by the Nakayama functor $\nu$ yields an isomorphism of varieties $\varP$ and the dually defined $\sC^{[-n,0]}_{\nu(\PP)}(\inj \Lambda)$. The dual of \Cref{lem:analogue_Zwara} with $\sC^{[-n,0]}_{\nu(\PP)}(\inj \Lambda)$ lets us find a short exact seuqence $0 \rightarrow Z \rightarrow M \oplus Z \rightarrow N \rightarrow 0$ with $Z \in \chnmod$. Using the $\Lambda^\text{op}$-duality, we can use \Cref{lem:ses_chnmod_to_chnproj} to replace the $Z$ by a complex in $\chnproj$.
\end{proof}

\begin{example}
\label{ex:trivial_deg_ses}
    We let $\Lambda$ be any finite dimensional algebra and reconsider \Cref{ex:trivial_deg}. We let $M = (P \xrightarrow{f} Q)$ and $N = (P \xrightarrow{0} Q)$. Recall that $M\leqdeg N$. For $Z = (P \xrightarrow{} 0)$, there is a short exact sequence in $\bchnproj$ of the form ${0 \rightarrow N \rightarrow M \oplus Z \rightarrow Z \rightarrow 0}.$ These also give the cover relations for $\Lambda = kA_1$ in \Cref{ex:A_1_chain_cx_deg}.
\end{example}

\subsection{The ext and hom order}

We can also define analogues of the ext and hom order for chain complex varieties.

\begin{definition}
\label{def:ext_order_ch_cx}
    Let $M,N \in \bchnproj$. We say $M \leqext N$ if there exist an integer $n \geq 1$ and a chain of $n$ short exact sequences $0 \rightarrow A_i \rightarrow B_i \rightarrow C_i \rightarrow 0$ with $B_1 = M$, for every $i \geq 2$, $B_i = A_{i-1} \oplus C_{i-1}$ and $N = A_n \oplus C_n$.
\end{definition}

\begin{remark}
    This is a well defined order on isomorphism classes of $\bchnproj$. The anti-symmetry is easiest seen by its relation to the degeneration order.
\end{remark}

\begin{definition}
\label{def:hom_order_ch_cx}
    Let $M,N \in \bchnproj$. We say $M \leqhom N$ if for any $C \in \bchnproj$, $\dim_k \Hom_{\bchnproj}(C,M) \leq \dim_k \Hom_{\bchnproj}(C,N)$.
\end{definition}

\begin{remark}
    This is a well defined order on isomorphism classes of $\bchnproj$ since two complexes have the same $\dim_k \Hom_{\bchnproj}(Z,-)$ for all $Z$ if and only if they are isomorphic, see \cite{B}.
\end{remark}

\begin{lemma}
\label{lem:ext_deg_hom_ch_cx}
    Let $M,N \in \bchnproj$. Then the following implications hold: 
    \begin{equation*}
        (M \leqext N) \Rightarrow (M \leqdeg N) \Rightarrow (M \leqhom N).
    \end{equation*}
\end{lemma}

\begin{proof}
    This is analogous to the representatrion variety arguments. The first implication follows from \Cref{thm:deg_and_ses_in_chnproj}: Let us assume that $M \leqext N$ is a covering relation. So $M = Y$ and $N = X \oplus Z$ and there is a short exact sequence in $\bchnproj$ of the form $0 \rightarrow X \rightarrow Y \rightarrow Z \rightarrow 0$. Then we can add the short exact sequence $0 \rightarrow Z \xrightarrow{\id} Z \rightarrow 0 \rightarrow 0$ to obtain a short exact sequence of the form 
    \begin{equation*}
        0 \rightarrow \underbrace{X \oplus Z}_{N} \rightarrow \underbrace{Y}_M \oplus Z \rightarrow Z \rightarrow 0.
    \end{equation*}
    By \Cref{thm:deg_and_ses_in_chnproj}, $M \leqdeg N$.
    
    For the second implication, let $M \leqdeg N$. By \Cref{thm:deg_and_ses_in_chnproj}, there is a short exact sequence $0 \rightarrow N \rightarrow M \oplus Z \rightarrow Z \rightarrow 0$. For any $C \in \bchnproj$, this gives rise to the long exact sequence 
    \begin{equation*}
        0 \rightarrow \Hom(C, N) \rightarrow \Hom(C, M \oplus Z) \rightarrow \Hom(C, Z) \rightarrow \Ext(C,N) \rightarrow \ldots
    \end{equation*}
    In particular, we see that $\dim \Hom(C, M \oplus Z) \leq \dim \Hom(C, N) + \dim \Hom(C, Z)$. Thus it follows that $\dim \Hom(C, M) \leq \dim \Hom(C, N)$ and $M \leqhom N$.
\end{proof}

\begin{example}
\label{ex:A_1_ext_hom_order} 
    While these three orders do not generally coincide, they do for $\Lambda = kA_1$ in \Cref{ex:A_1_chain_cx_deg}. We recall that the degeneration order is generated by covering relations of the form $(k \xrightarrow{\id} k) \leqdeg (k \xrightarrow{0} k)$. Since there is a short exact sequence of the form $0 \rightarrow (0 \rightarrow k) \rightarrow (k \xrightarrow{\id} k) \rightarrow (k \rightarrow 0) \rightarrow 0$, we also have that $(k \xrightarrow{\id} k) \leqext (k \xrightarrow{0} k)$ and the two orders coincide.
    
    For the hom order, we use that we understand the hom spaces between indecomposables. We can show that for two complexes $X$ and $Y$, we have $X \leqhom Y$, if for each acyclic indecomposable, $X$ contains more copies of it than $Y$. This means that any hom order relation factors through the covering relations of the degeneration order and these two orders coincide.
\end{example}

\begin{remark}
    We will see an example where the ext and degeneration order are not identical at the end of \Cref{sect:Kb}.
\end{remark}

\begin{example}
    The degeneration and hom order are not in general identical. For representation varieties, an example is given in  \cite[Example 6.3]{Z_survey}. We are able to extend this to chain complexes in the following way. Let $\Lambda$ be a finite dimensional algebra given by the quiver $
    \begin{tikzcd}
    	1 & 2 & 3
    	\arrow["a"', shift right, from=1-2, to=1-1]
    	\arrow["b", shift left, from=1-2, to=1-1]
    	\arrow["c"', shift right, from=1-3, to=1-2]
    	\arrow["d", shift left, from=1-3, to=1-2]
    \end{tikzcd}$ and ideal generated by $da, cb, ca-cd$. We denote with $M = P_3$ the projective at vertex $3$ and $L = I_2$ the injective at vertex $2$. We let $S_2$ be the simple at vertex $2$. Further we define two modules, $U$ and $V$, as
    \begin{equation*}
        U \colon \begin{tikzcd}[ampersand replacement=\&]
        	k \& k \& 0,
        	\arrow["1"', shift right, from=1-2, to=1-1]
        	\arrow["0", shift left, from=1-2, to=1-1]
        	\arrow[shift right, from=1-3, to=1-2]
        	\arrow[shift left, from=1-3, to=1-2]
        \end{tikzcd}
        \qquad
        V \colon \begin{tikzcd}[ampersand replacement=\&]
        	0 \& k \& k.
        	\arrow[shift right, from=1-2, to=1-1]
        	\arrow[shift left, from=1-2, to=1-1]
        	\arrow["1"', shift right, from=1-3, to=1-2]
        	\arrow["1", shift left, from=1-3, to=1-2]
        \end{tikzcd}
    \end{equation*}
    Zwara shows that $M \leqhom U \oplus V$ in the following way. There are short exact sequences in $\mod \Lambda$ of the form $0 \rightarrow L \rightarrow M \oplus S_2 \rightarrow V \rightarrow 0$ and $0 \rightarrow U \rightarrow L \rightarrow S_2 \rightarrow 0$. So $M \oplus S_2 \leqext L \oplus V$ and $L \leqext U \oplus S_2$. By transitivity, we obtain that $M \oplus S_2 \leqext U \oplus V \oplus S_2$. Since the ext order implies the hom order, $M \oplus S_2 \leqhom U \oplus V \oplus S_2$ and we may cancel $S_2$ on both sides to show that $M \leqhom U \oplus V$. However, a geometric argument shows that $M$ does not degenerate to $U \oplus V$.
    
    This argument can be lifted to chain complexes. We let $P_X$ be the minimal projective resolution of a module $X$. By the horseshoe lemma, the above short exact sequences in $\mod \Lambda$ give rise to short exact sequences in $\bchnproj$ of the form
    $0 \rightarrow P_L \rightarrow P_M \oplus P_{S_2} \oplus C_1 \rightarrow P_V \rightarrow 0$ and $0 \rightarrow P_U \rightarrow P_L \oplus C_2 \rightarrow P_{S_2} \rightarrow 0$
    where $C_1$ and $C_2$ are acyclic complexes. Thus the same argument as before will give us that $P_M \oplus C_1 \oplus C_2 \leqhom P_U \oplus P_V$. It remains to show that $P_U \oplus P_V$ is not a degeneration of $P_M \oplus C_1 \oplus C_2$. 
    
    We do this geometrically. We know that $P_M = (P_3)$ concentrated in degree 0, $P_U = (P_1 \xrightarrow{b} P_2)$ in degree $-1$ and $0$ and $P_V = (P_1 \xrightarrow{a+b} P_2 \xrightarrow{c-d} P_3)$ in degrees $-2$ to $0$  where $a,b,c,d$ denote the morphisms induced by the arrows of the same name.
    We denote $P_M \oplus C_1 \oplus C_2$ by $\widetilde{P_M}$. Since it lies in the same variety $\varP$ as $P_U \oplus P_V$, it is up to isomorphism of the form
    \begin{equation*}
        \widetilde{P_M} = (P_M \oplus C_1 \oplus C_2) = (P_1 \xrightarrow{\pmat{1 \\0}} P_1 \oplus P_2 \xrightarrow{\pmat{0 & 1 \\ 0 & 0}} P_2 \oplus P_3).
    \end{equation*}
    An element $g$ in $\GP$ is of the form
    \begin{equation*}
        (\lambda, \pmat{\mu_1 & 0 \\ \psi & \mu_2}, \pmat{\eta_1 & 0 \\ \varphi & \eta_2})
    \end{equation*}
    where $\lambda, \mu_1, \mu_2, \eta_1, \eta_2 \in k^*$, $\varphi = \varphi_1 c + \varphi_2 d$ and $\psi = \psi_1 a + \psi_2 b$. Then the elements in the orbit of $\widetilde{P_M}$ are of the form
    \begin{equation*}
        g.\widetilde{P_M} = (
        P_1 \xrightarrow{\lambda^{-1} \pmat{\mu_1 \\ \psi}} P_1 \oplus P_2 \xrightarrow{\pmat{-\eta_1\mu_1^{-1}\mu_2^{-1}\psi\quad & \eta_1\mu_2^{-1} \\-\mu_1^{-1}\mu_2^{-1}\varphi\psi & \mu_2^{-1}\varphi}} P_2 \oplus P_3 ).
    \end{equation*}
    In comparison, $P_U \oplus P_V$ is given as
    \begin{equation*}
        P_1 \xrightarrow{\pmat{0 \\ a+b}} P_1 \oplus P_2 \xrightarrow{\pmat{b & 0 \\0 & c-d}} P_2 \oplus P_3.
    \end{equation*}
    The problem lies with the $2\times 2$-matrix. While closing the orbit of $\widetilde{P_M}$ does allow the non-zero entry $\eta_1\mu_2^{-1}$ to become zero, this implies at least one of the top left and bottom right entries also become $0$. So $P_U \oplus P_V$ is not contained in the orbit closure of $\widetilde{P_M}$ and thus there is no degeneration.
\end{example}

\subsection{Other examples}
\label{sect:examples}

We discuss some additional examples to show interesting behaviours not exhibited by $\Lambda = kA_1$.

\begin{example}
\label{ex:A_2_preproj}
    Let $\Lambda = kQ/I$ be the preprojective algebra of type $A_2$ where $Q = 
    \begin{tikzcd}
    	1 & 2
    	\arrow["a", shift left, from=1-1, to=1-2]
    	\arrow["{a^*}", shift left, from=1-2, to=1-1]
    \end{tikzcd}$ and $I = \langle aa^*,a^*a \rangle$.
    There are two projective indecomposables, $P_1 = \Lambda e_1$ and $P_2 = \Lambda e_2$. Note that the homomorphism spaces between $P_1$ and $P_2$ are one dimensional. We let $f \colon P_1 \rightarrow P_2$ and $g \colon P_2 \rightarrow P_1$ be non-zero morphisms. We define $\PP = (P_1\oplus P_2, P_1 \oplus P_2)$ supported in degree $-1$ and $0$. Then 
    \begin{equation*}
        \varP = \End(P_1\oplus P_2) \cong \left\{ \pmat{a \amsamp b \\ c \amsamp d} \mid a,b,c,d \in k \right\}
    \end{equation*}
    and 
    \begin{equation*}
        \GP = \Aut(P_1\oplus P_2)^2 \cong \left\{ \pmat{a \amsamp b \\ c \amsamp d} \mid a,b,c,d \in k, ad \neq 0 \right\}^2.
    \end{equation*}
    The group structure on $\GP$ is not the usual matrix multiplications since $f \circ g = g \circ f = 0$. Instead the multiplication is defined as
    \begin{equation*}
        \pmat{a \amsamp b \\ c \amsamp d} \circ \pmat{a' \amsamp b' \\ c' \amsamp d'} = \pmat{aa' \amsamp ab'+bd' \\ ca'+dc' \amsamp dd'}.
    \end{equation*}
    This multiplication is also the one used when conjugating for the group action. There are seven orbits represented by, respectively, $\pmat{0 \amsamp 0 \\ 0 \amsamp 0}$, $\pmat{0 \amsamp 1 \\ 0 \amsamp 0}$, $\pmat{0 \amsamp 0 \\ 1 \amsamp 0}$, $\pmat{0 \amsamp 1 \\ 1 \amsamp 0}$, $\pmat{1 \amsamp 0 \\ 0 \amsamp 0}$, $\pmat{0 \amsamp 0 \\ 0 \amsamp 1}$ and $\pmat{1 \amsamp 0 \\ 0 \amsamp 1}$.

    The poset of degenerations of $\varP$ is the following:
    \[\begin{tikzcd}[sep=small]
    	& \begin{array}{c} \pmatquiver{0 \amsamp 0 \\ 0 \amsamp 0} \end{array} & \\
    	\begin{array}{c} \pmatquiver{0 \amsamp 0 \\ 1 \amsamp 0} \end{array} && \begin{array}{c} \pmatquiver{0 \amsamp 1 \\ 0 \amsamp 0} \end{array} \\
    	& \begin{array}{c} \pmatquiver{0 \amsamp 1 \\ 1 \amsamp 0} \end{array} \\
    	\begin{array}{c} \pmatquiver{1 \amsamp 0 \\ 0 \amsamp 0} \end{array} && \begin{array}{c} \pmatquiver{0 \amsamp 0 \\ 0 \amsamp 1} \end{array} \\
    	& \begin{array}{c} \pmatquiver{1 \amsamp 0 \\ 0 \amsamp 1} \end{array}
    	\arrow[no head, from=1-2, to=2-1]
    	\arrow[no head, from=1-2, to=2-3]
    	\arrow[no head, from=2-1, to=3-2]
    	\arrow[no head, from=2-3, to=3-2]
    	\arrow[no head, from=3-2, to=4-1]
    	\arrow[no head, from=3-2, to=4-3]
    	\arrow[no head, from=4-1, to=5-2]
    	\arrow[no head, from=4-3, to=5-2]
    \end{tikzcd}\]
\end{example}

\begin{example}
\label{ex:A_2}
    Let $\Lambda = k(\begin{tikzcd}
        1 & 2
        \arrow["a", from=1-1, to=1-2]
    \end{tikzcd})$. There are two projective indecomposables, $P_1 = \Lambda e_1$ and $P_2 = \Lambda e_2$. Note that there is a non-zero morphism from $P_2$ to $P_1$. Let $m \geq 1$. We define $\PP_m = (P_1^m\oplus P_2^m, P_1^m \oplus P_2^m)$ supported in degree $-1$ and $0$. Then 
    \begin{equation*}
        \varP = \End(P_1^m\oplus P_2^m) \cong \left\{ \pmat{A \amsamp B \\ 0 \amsamp C} \mid A,B,C \in \Mat_{m \times m}(k) \right\}
    \end{equation*}
    and 
    \begin{equation*}
        \GP = \Aut(P_1^m\oplus P_2^m)^2 \cong \left\{ \pmat{A \amsamp B \\ 0 \amsamp C} \mid A,B \in \GL_m(k), C \in \Mat_{m \times m}(k) \right\}^2.
    \end{equation*}
    acting via two-sided conjugation by the usual matrix multiplication.
    The orbits are characterised by the cohomology of the associated complexes. There are three types of indecomposable complexes: the indecomposable stalks, the indecomposable acyclics and $(P_2 \xrightarrow{f} P_1)$. The indecomposable stalks have themselves as cohomology. The acyclics have a zero cohomology and $(P_2 \xrightarrow{f} P_1)$ has a cohomology concentrated in degree $0$ which is the simple at $1$. Clearly, the cohomology completely determines the orbit. In particular, for a complex $X$ in $\varP$, the cohomology is going to be of the form $H^*(X) = (P_1^a \oplus P_2^b, P_1^c \oplus P_2^d \oplus S_1^e)$ where the $(a,b,c,d,e)$ are unique for a given isomorphism class. The entries $a$ to $d$ correspond to the multiplicities of the corresponding stalk complexes, the entry $e$ to the multiplicity of $(P_2 \xrightarrow{f} P_1)$. Since this implicitly also encodes the acyclics, we see that $b+e = d$ and $c+e = a$. Thus the orbits are characterised by tuples $(b,c,e)\in \bbN^3$ satisfying $b+ e \leq m$ and $c+e \leq m$. Such a tuple $(b,c,e)$ corresponds to the orbit of 
    \begin{equation*}
    \scriptstyle
        P_1[-1]^{c+e} \oplus P_1[0]^{c} \oplus P_2[-1]^{b} \oplus P_2[0]^{b+e} \oplus (P_2 \xrightarrow{f} P_1)^e \oplus (P_1 \xrightarrow{\id} P_1)^{m-c-e}  \oplus (P_2 \xrightarrow{\id} P_2)^{m-b-e}.
    \end{equation*}
    In particular, the orbits are enumerated by the square pyramidal numbers\footnote{\url{https://oeis.org/A000330}} starting at $5$ for $m=1$.

    We let $m = 1$. Then the degeneration order is given as
    \[\begin{tikzcd}[sep=small]
    	& {(1,0,1)} &&&& \begin{array}{c} \pmatquiver{0 \amsamp 0 \\0 \amsamp 0} \end{array} & \\
    	& {(0,1,0)} &&&& \begin{array}{c} \pmatquiver{0 \amsamp 1 \\0 \amsamp 0} \end{array} \\
    	{(0,0,1)} && {(1,0,0)} & \longleftrightarrow & \begin{array}{c} \pmatquiver{1 \amsamp 0 \\0 \amsamp 0} \end{array} && \begin{array}{c} \pmatquiver{0 \amsamp 0 \\0 \amsamp 1} \end{array} \\
    	& {(0,0,0)} &&&& \begin{array}{c} \pmatquiver{1 \amsamp 0 \\0 \amsamp 1} \end{array}
    	\arrow[no head, from=1-2, to=2-2]
    	\arrow[no head, from=1-6, to=2-6]
    	\arrow[no head, from=2-2, to=3-1]
    	\arrow[no head, from=2-2, to=3-3]
    	\arrow[no head, from=2-6, to=3-5]
    	\arrow[no head, from=2-6, to=3-7]
    	\arrow[no head, from=3-1, to=4-2]
    	\arrow[no head, from=3-3, to=4-2]
    	\arrow[no head, from=3-5, to=4-6]
    	\arrow[no head, from=3-7, to=4-6]
    \end{tikzcd}\]
    There are two types of covering degenerations. Those given by \Cref{ex:trivial_deg} and those generated by $M \leqdeg N$ with $M = (P_1 \xrightarrow{\id} P_1) \oplus (0 \xrightarrow{} P_2) $ and $N = (P_1 \xrightarrow{f} P_2) \oplus (0 \xrightarrow{} P_1)$. Note that the dual degeneration with $P_1$ and $P_2$ swapped also exists. We can describe these degenerations via short exact sequences as well. We let $Z = P_1 \xrightarrow{} 0$. Then there is a short exact sequence in $\bchnproj$ of the form $ 0 \rightarrow N \rightarrow M \oplus Z \rightarrow Z \rightarrow 0$.
    
    For $m = 2$, we have the following degeneration order:
    \[\begin{tikzcd}[sep=small]
    	&& {(2,0,2)} && \\
    	&& {(1,1,1)} \\
    	{(1,0,2)} && {(0,2,0)} && {(2,0,1)} \\
    	& {(0,1,1)} && {(1,1,0)} \\
    	{(0,0,2)} && {(1,0,1)} && {(2,0,0)} \\
    	&& {(0,1,0)} \\
    	& {(0,0,1)} && {(1,0,0)} \\
    	&& {(0,0,0)}
    	\arrow[no head, from=1-3, to=2-3]
    	\arrow[no head, from=2-3, to=3-1]
    	\arrow[no head, from=2-3, to=3-3]
    	\arrow[no head, from=2-3, to=3-5]
    	\arrow[no head, from=3-1, to=4-2]
    	\arrow[no head, from=3-3, to=4-2]
    	\arrow[no head, from=3-3, to=4-4]
    	\arrow[no head, from=3-5, to=4-4]
    	\arrow[no head, from=4-2, to=5-1]
    	\arrow[no head, from=4-2, to=5-3]
    	\arrow[no head, from=4-4, to=5-3]
    	\arrow[no head, from=4-4, to=5-5]
    	\arrow[no head, from=5-1, to=7-2]
    	\arrow[no head, from=5-3, to=6-3]
    	\arrow[no head, from=5-5, to=7-4]
    	\arrow[no head, from=6-3, to=7-2]
    	\arrow[no head, from=6-3, to=7-4]
    	\arrow[no head, from=7-2, to=8-3]
    	\arrow[no head, from=7-4, to=8-3]
    \end{tikzcd}\]
    We remark that the poset for $m=1$ appears a subposet at the bottom.
\end{example}

\biblio

\section{The homotopy category}
\label{sect:Kb}

In this section, we extend our constructions and results so far to consider complexes up to homotopy. We let 
\begin{equation*}
    \bhomproj = \bchnproj/\{\text{null-homotopic maps}\}
\end{equation*}
be the homotopy category of bounded chain complexes whose underlying modules are finite-dimensional projectives. This is a triangulated category with distinguished triangles of the form
\begin{equation*}
    X \xrightarrow{f} Y \rightarrow C(f) \rightarrow X[1]
\end{equation*}
where $C(f)$ denotes the cone of $f$. Note that 
\begin{equation*}
    0 \rightarrow Y \rightarrow C(f) \rightarrow X[1] \rightarrow 0.
\end{equation*}
is a short exact sequence in $\bchnproj$. This implies the following well-known lemma.
\begin{lemma}
\label{lem:triangles_and_ses}
Let $X \xrightarrow{f} Y \rightarrow Z \rightarrow X[1]$ be a distinguished triangle in $\bhomproj$. Then there is an isomorphic triangle of the form $X \xrightarrow{f} Y' \rightarrow Z \rightarrow X[1]$ where $0 \rightarrow X \rightarrow Y' \rightarrow Z \rightarrow 0$ is a short exact sequence in $\bchnproj$.

Moreover, every short exact sequence $0 \rightarrow X \rightarrow Y \rightarrow Z \rightarrow 0$ in $\bchnproj$ induces a triangle in $\bhomproj$.
\end{lemma}

\subsection{Ind-varieties of the homotopy category}
\label{subsect:var_of_Kb}

We first need to introduce the right structure to go from varieties of chain complexes to a notion modelling the homotopy category.

\subsubsection{Ind-varieties}
\label{subsec:ind_var}

We will consider some colimits of varieties, called ind-varieties. We follow the lecture notes \cite{Ric}. A similar definition of ind-varieties is discussed in \cite{BD}. The colimits we look at are filtered but not necessarily linear. However, colimits over $\bbN$ provide some nice examples, as in \cite{FK}. For more details on filterd colimits in any category, we refer to \cite{ML}.

Recall that a scheme $X$ can be interpreted as a functor via the Yoneda embedding to its functor of points $\Hom(-,X) \colon \AffSch^{\text{op}} \rightarrow \Sets$ where $\AffSch$ is the category of affine schemes and $Sets$ the category of sets.

\begin{definition}
\label{def:ind_scheme}
A strict $\aleph_0$-ind-scheme is a functor $X \colon \AffSch^{\text{op}} \rightarrow \Sets$ such that $X \cong \colim_{i \in I} X_i$ is a filtered colimit where $I$ is countable, $X_i$ is a scheme and for all $i < j$, the transition map $\varphi_{ij} \colon X_i \rightarrow X_j$ is a closed immersion. We denote the morphism into the colimit as $\varphi_i \colon X_i \rightarrow X$.

It is called an (affine) ind-variety if every $X_i$ is an (affine) variety.
\end{definition}

\begin{remark}
\label{rmk:ind_countable}
    Since $I$ is countable and the colimit is filtered, $X$ can also be expressed as a colimit over a suitable linear order, see \cite{Ric}. 
\end{remark}

\begin{example}[\cite{Ric}]
\label{ex:A_infty}
    An example for an affine ind-variety is $\bbA^\infty = \colim_{n\in \bbN} \bbA^n$ where $\varphi_{n, n+1} \colon \bbA^n \rightarrow \bbA^{n+1}, (x_1, \ldots, x_n) \mapsto (x_1, \ldots, x_n, 0)$. 
\end{example}

\begin{remark}
\label{rmk:ind_pts}
    Let $X$ be an ind-variety. For any point $x \in X$, there is some $i \in I$ and $x_i \in X_i$ such that $\varphi_i(x_i) = x$.
\end{remark}

The topology of an ind-scheme $X$ can be understood in terms of the $X_i$, as shown in \cite[Lemma 1.12]{Ric}.
\begin{lemma}
\label{lem:colimit_topology}  
    Any ind-scheme $X$ carries the colimit topology, i.e. $U \subset X$ is closed (resp. open) if and only if for every $i \in I$, $\varphi_i^{-1}(U)$ is closed (resp. open) in $X_i$.
\end{lemma}

\begin{definition}
\label{def:ind-morphism}
    Let $X \cong \colim_{i\in I} X_i$ and $Y \cong \colim_{j \in J} Y_j$ be two ind schemes. We call the morphisms in the category of functors from $\AffSch^{\text{op}}$ to $\Sets$ from $X$ to $Y$ ind-morphisms. Any such morphism can be written (up to changement of $I$) as a system of morphisms of schemes $f_{ij}\colon X_i \rightarrow Y_j$, i.e. for all $i<i'$, $j<j'$, $f_{i'j'}\varphi_{ii'} = \varphi_{jj'}f_{ij}$. In particular, such a system gives rise to an ind-morphism.
\end{definition}

\begin{remark}
\label{rmk:ind_products}  
    The product of two ind-schemes $X \cong \colim_{i\in I} X_i$ and $Y \cong \colim_{just \in J} Y_j$ is generally defined as $X \times Y \cong \colim_{(i,j)\in I \times J} X_i \times Y_j$. If $I = J$ then we can use the fact that $I \cong \{(i,i) \mid i \in I\} \subset I \times I$ is final to show that $X \times Y \cong \colim_{i\in I} X_i \times Y_i$.
\end{remark}

\begin{definition}
\label{def:ind_group}
    An ind-group is an ind-variety $G$ with an internal group structure such that $m \colon G \times G \rightarrow G, (g,h) \mapsto gh$ and $i \colon G \rightarrow G, g \mapsto g^{-1}$ are ind-morphisms.
\end{definition}

\begin{example}
\label{ex:ind-group}
    Let $G = \colim_{i \in I} G_i$ be an ind-variety where $G_i$ is an algebraic group and the inclusion $G_{i} \subset G_{j}$ an injection of groups for every $i<j$. Then $G$ is an ind-group. 
    A concrete example of this structure is $\GL_\infty(k) = \colim_{n \in \bbN} \GL_n(k)$ where $\varphi_{n, n+1} \colon \GL_n(k) \rightarrow \GL_{n+1}(k), f \mapsto \pmat{f \amsamp 0 \\ 0 \amsamp 1}$.
\end{example}

\begin{definition}
\label{def:ind_action}
    An action of an ind-group $G$ on an ind-variety $X$ is a homomorphism $\rho \colon G \rightarrow \Aut(X)$ such that $G \times X \rightarrow X, (g,x) \mapsto g.x := \rho(g)x$ is an ind-morphism.
\end{definition}

\begin{example}
\label{ex:ind-action}
    Let $G$ be as in \Cref{ex:ind-group} and let $X \cong \colim_{i\in I} X_i$ be an ind-variety. Let us assume that for every $i \in I$ there is a $G_i$-action on $X_i$, given by $\rho_i \colon G_i \times X_i \rightarrow X_i$ compatible with the colimit structure, i.e. such that for every $i < j$ the following diagram commutes:
    \[\begin{tikzcd}
    	{G_i \times X_i} & {X_i} \\
    	{G_{j} \times X_{j}} & {X_{j}}
    	\arrow["{\rho_i}", from=1-1, to=1-2]
    	\arrow["{\varphi_{ij} \times \varphi_{ij}}"', from=1-1, to=2-1]
    	\arrow["{\varphi_{ij}}", from=1-2, to=2-2]
    	\arrow["{\rho_{j}}"', from=2-1, to=2-2]
    \end{tikzcd}\]
    Then there is unique action of $G$ on $X$ given by the universal property of the colimit.

    As a concrete example, this is how one can obtain the action of $\GL_\infty$ on $\bbA^\infty$.
\end{example}

\begin{lemma}
\label{lem:ind_orbits}
Let $G$ be an ind-group acting on an ind-variety $X$ as in \Cref{ex:ind-action}. Let $x \in X$. Let $i \in I$ and $x_i \in X_i$ such that $\varphi_i(x_i) = x$. For any $j > i$, we set $x_j := \varphi_{ij}(x_i)$. We denote the orbit of $x$ in $X$ by $\sO_x$ and the orbit of $x_j$ in $X_j$ by $\sO_{x_j}$. We assume that $\varphi_j^{-1}(\sO_x) = \sO_{x_j}$. The following hold.
\begin{enumerate}
    \item The $\sO_{x_j}$ form a filtered system. Moreover, $\sO_x = \colim_{j \in I, j > i} \sO_{x_j}$.
    \item $\overline{\sO_x} = \colim_{j \in I, j > i} \overline{\sO_{x_j}}$.
\end{enumerate}
\end{lemma}

\begin{proof}
We prove (1) first. By our assumptions, for any $\ell > j > i$, we have an inclusion $\varphi_{j\ell} (\sO_{x_j}) \subset \sO_{x_\ell}$. So these orbits form a filtered system. Let $y \in \sO_x$. So there exists a $g \in G$ such that $g.x = y$ in $X$. Then there exist $j > i$, $y_j \in X_j$ and $g_j \in G$ such that we have $\varphi_j(y_j) = y$, $\varphi_j(g_j) = g$ and $g_j.x_j = y_j$ in $X_j$. In particular, we have that $y \in \colim_{j \in I, j > i} \sO_{x_j}$ . The inverse inclusion is clear.

Now we prove (2). Clearly, $\overline{\sO_x}$ is closed in $X$. By \Cref{lem:colimit_topology}, so is $\colim_{j \in I, j > i} \overline{\sO_{x_j}}$. Moreover,we have an inclusion $\sO_x = \colim_{j \in I, j > i} \sO_{x_j} \subset \colim_{j \in I, j > i} \overline{\sO_{x_j}}$, so we obtain the inclusion of its closures $\overline{\sO_x} \subset \colim_{j \in I, j > i} \overline{\sO_{x_j}}$. On the other hand, we know that there is an inclusion $\sO_{x_j} = \varphi_j^{-1}(\sO_x) \subset \varphi_j^{-1}(\overline{\sO_x})$. So we again get the closed version, $\overline{\sO_{x_j}} \subset \varphi_j^{-1}(\overline{\sO_x})$. Applying $\varphi_j$ gives us the inclusion $\varphi_j(\overline{\sO_{x_j}}) \subset \overline{\sO_x}$. As a consequence, we have that $\colim_{j \in I, j > i} \overline{\sO_{x_j}} \subset \overline{\sO_x}$.
\end{proof}

\subsubsection{Application to the homotopy category}

Note that in $\bhomproj$, $X \cong Y$ if and only if there exist chain complexes $C, X_0, Y_0 \in \bchnproj$ with $X_0$ and $Y_0$ null-homotopic such that there are isomorphisms in $\bchnproj$ of the form $X \cong C \oplus X_0$ and $Y \cong C \oplus Y_0$.. In particular, $X \oplus Y_0 \cong Y \oplus X_0$ in $\bchnproj$. This leads us to defining ind-varieties associated to the homotopy category as colimits of chain complex varieties. 

Let us assume that $\Lambda$ has $p$ pairwise-different indecomposable projectives $P_1, \ldots, P_p$ up to isomorphism. A null-homotopic chain complex in $\bhomproj$ is isomorphic to a finite direct sum of complexes of the form 
\begin{equation*}
    \cdots \rightarrow 0 \rightarrow \underset{q}{P_i} \xrightarrow{\id} \underset{q+1}{{P_i}} \rightarrow 0 \rightarrow \cdots
\end{equation*}
where the $P_i$ are in cohomological degrees $q$ and $q+1$. We denote this complex by $P_{iq}$.

We let $v = (v^q)_{q \in \bbZ} \in (\bbN^p)^\bbZ$ be a finitely supported point, i.e. $v^q = 0$ for all but finitely many $q$. Here, $v^q = (v^q_i) \in \bbN^p$. To $v$, we associate $P(v) = (\bigoplus_{1 \leq i \leq p} P_i^{v^q_i})_{q \in \bbZ}$. Note that any set of underlying projectives of a chain complex can be uniquely obtained this way. We define the $g$-vector of $v$ as $g(v) = g(P(v)) = [P(v)] = \sum_{q \in \bbZ} (-1)^q v^q$, the element associtated to $P(v)$ in the Grothendieck group of $\bhomproj$. We define an ind-variety for every $g$-vector.

Let $e_{iq} \in (\bbN^p)^\bbZ$ be the element associated to $P_{iq}$. We define the morphisms $\varphi_{v,iq}$ and $\psi_{v,iq}$ as follows
\begin{equation*}
    \begin{aligned}
        \varphi_{v,iq} \colon \sC_{P(v)} &\hookrightarrow \sC_{P(v + e_{iq})},\\ X &\mapsto X \oplus P_{iq} 
    \end{aligned}
    \qquad \qquad
    \begin{aligned}
        \psi_{v,iq} \colon \sG_{P(v)} &\hookrightarrow \sG_{P(v + e_{iq})}.\\ f &\mapsto f \oplus \id_{P_{iq}}
    \end{aligned}
\end{equation*}
For $v,w \in (\bbN^p)^\bbZ$, we say $v \leq w$ if $w$ can be obtained from $v$ by adding a finite number of vectors of type $e_{iq}$.

\begin{definition}
    Let $g \in \bbN^p$. The ind-variety for $g$ is defined as 
    \begin{equation*}
        \bvarPhom = \colim_{\substack{v \in (\bbN^p)^\bbZ\\g(v) = g}} C_{P(v)}
    \end{equation*}
    where the colimit structure is given by the $\varphi_{v,iq}$ for all $v \in (\bbN^p)^\bbZ$ such that $g(v) = g$, $q \in \bbZ$ and $1 \leq i \leq p$. We denote the morphisms into the colimit as $\varphi_v \colon C_{P(v)} \rightarrow \bvarPhom$.
    
    We define the ind-group as
    \begin{equation*}
        \bGhom = \colim_{\substack{v \in (\bbN^p)^\bbZ\\g(v) = g}} \sG_{P(v)}.
    \end{equation*}
    where the colimit structure is given by the $\psi_{v,iq}$. We denote the morphisms into the colimit as $\psi_v \colon \sG_{P(v)} \rightarrow \bGhom$.
    
    For an interval $[a,b]$ we define the ind-variety for $g$ restricted to $[a,b]$ as 
    \begin{equation*}
        \ivarPhom{a}{b} = \colim_{\substack{v \in (\bbN^p)^{[a,b]}\\g(v) = g}} C_{P(v)}
    \end{equation*}
    where $(\bbN^p)^{[a,b]}\subset (\bbN^p)^\bbZ$ is the subset of $v$ supported on the interval $[a,b]$. We also define the ind-group
    \begin{equation*}
        \iGhom{a}{b} = \colim_{\substack{v \in (\bbN^p)^{[a,b]}\\g(v) = g}} \sG_{P(v)}.
    \end{equation*}
    We use the same notation as before for morphisms into these colimits.
\end{definition}

\begin{remark}
These are all filtered colimits, so they are well-defined ind-varieties and ind-groups. In the following, for statements holding both for the general case or the restriction to an interval, we will use the notation $\varPhom$, resp. $\Ghom$.
\end{remark}

\begin{remark}
    We emphasize that these ind-varieties do not appear in \cite{JSZ_ch_cx}. In their notation, $C_{P(v)}$ is $comproj^v$. They also consider topological spaces of the same name where $v$ is only right bounded but this is neither $\bvarPhom$ nor $\ivarPhom{a}{b}$.
\end{remark}

\begin{remark}
    By construction, the following hold, cf. \Cref{rmk:ind_pts} and \Cref{ex:ind-action}:
    \begin{enumerate}
        \item For a point $M \in \varPhom$, there exists a $v \in (\bbN^p)^\ast$ and a complex $M_v \in \sC_{P(v)}$ such that $\varphi_v(M_v) = M$. For any $w \geq v$, we define $M_w = \varphi_{vw}(M_v)$.
        \item For a point $g \in \varPhom$, there exists a $v \in (\bbN^p)^\ast$ and a morphism $g_v \in \sG_{P(v)}$ such that $\psi(g_v) = g$. For any $w \geq v$, we define $g_w = \psi_{vw}(g_v)$.
        \item The ind-group $\Ghom$ acts on $\varPhom$.
    \end{enumerate}
\end{remark}

In the following, we list some important properties of $\varPhom$ and $\Ghom$.

\begin{proposition}
\label{prop:properties_ind_var}
    Let $g \in \bbN^p$. The following hold.
    \begin{enumerate}
        \item \label{item:homotopy} For a point $M \in \varPhom$, $\sO_M$ is its homotopy class where $M$ is seen as a complex.
        \item \label{item:prereq} For a point $M \in \varPhom$ associated to some $M_v \in C_{P(v)}$, for any $w \geq v$, $\varphi_w^{-1}(\sO_M) = \sO_{M_w}$.
        \item \label{item:lemma} For a point $M \in \varPhom$ associated to some $M_v \in C_{P(v)}$, $\sO_M = \colim_{w \geq v} \sO_{M_w}$ and $\overline{\sO_M} = \colim_{w \geq v} \overline{\sO_{M_w}}$.
        \item \label{item:connected} The ind-variety $\varPhom$ is connected as a topological space.
    \end{enumerate}
\end{proposition}

\begin{proof}
\begin{enumerate}
    \item 
    Let $N \in \varPhom$ such that $N \in \sO_M$, i.e. there exists a $g \in \Ghom$ such that $g.M = N$. Then there exist $v$,$w$,$x$ and $M_v \in C_{P(v)}$, $N_w \in C_{P(w)}$ and $g_x \in \sG_{P(x)}$ such that $M = \varphi_v(M_v)$, $N = \varphi_w(N_w)$ and $g = \psi_x(g_x)$. Since the colimit is filtered, there exists a $y$ greater than $v$, $w$ and $x$ such that $g_y.M_y = N_y$. So $M_y$ and $N_y$ are isomorphic complexes. Note that $M_y = M_v \oplus Q_M$ and $N_y = N_w \oplus Q_N$ where $Q_M$ and $Q_N$ are acyclic. So $M_v$ and $N_w$ are homotopic. The inverse is clear.
    \item 
    Let $N \in \varPhom$ such that $N \in \sO_M$. Assume that for some $w \geq v$, there exists $N_w \in C_{P(w)}$ such that $N = \varphi_w(N_w)$. We need to show that $N_w$ is in the orbit of $M_w$, i.e. that $N_w$ is isomorphic to $M_w$ as a chain complex. Since $N \in \sO_M$, by (\ref{item:homotopy}) $N_w$ is homotopic to $M_w$. Since they lie in the same $C_{P(w)}$, they have the same underlying modules. Now we apply that chain complexes with the same underlying modules are homotopic if and only if they are isomorphic.
    \item  
    Due to construction and point (\ref{item:prereq}), \Cref{lem:ind_orbits} holds.
    \item  
    Let $M, N \in \varPhom$. By the same argument as in the beginning of (\ref{item:homotopy}), there is a $v$ such that $M$ and $N$ correspond to $M_v, N_v \in \sC_{P(v)}$. Let $C \in \sC_{P(v)}$ be the complex with zero morphisms in $\sC_{P(v)}$. Then $M_v \leqdeg C$ and $N_v \leqdeg C$. So $\overline{\sO_M}$ intersects $\overline{\sO_N}$. In particular, $M$ and $N$ are connected.
\end{enumerate}
\end{proof}

\begin{definition}
\label{def:deg_homotopy}
Let $M,N \in \varPhom$. We say that $M$ degenerates to $N$ or that $N$ is a degeneration of $M$, written $M \leqdeghom N$, if there is an inclusion $\sO_N \subset \overline{\sO_M}$ in $\varPhom$.
\end{definition}

\begin{remark}
\label{rmk:deg_homotopy}
    By \Cref{prop:properties_ind_var} (\ref{item:lemma}), $M \leqdeghom N$ if and only if there is some $v$ such that $M_v \leqdeg N_v$. Moreover if we have already chosen representatives $M_v$ and $N_w$, $M \leqdeghom N$ if and only if we can add acyclic complexes $Q_M$ and $Q_N$ (if applicable, supported on $[a,b]$) such that $M_v \oplus Q_M \leqdeg N_w \oplus Q_N$.
\end{remark}

\begin{example}
\label{ex:A_1_homotopy}
    We return to $\Lambda = kA_1$ and \Cref{ex:A_1_chain_cx_deg}. We describe $\sK^{[-n,0]}_0$. Note that the $\PP$ with $d(\PP) = (m,2m, \ldots, 2m, m)$ we considered before form a final set in our colimit. So it is enough to understand them to calculate the colimit. We consider the same vectors $a$ as before. In the colimit, condition (1) becomes obsolete as there is always a $m >> 0$ such that $a \leq d(\PP)$. The same is true for condition (2)(b). So an element in the colimit has the form $a \in \bbN^{n+1}$ with $\sum_{-n \leq i \leq 0} (-1)^i a_i = 0$. The degeneration order is still described by the same cover relations as before.
    
    Let $n = 1$. Then the poset of degenerations is isomorphic to $(\bbN, \leq)$ and has the form
    \[\begin{tikzcd}
    	\vdots && \vdots \\
    	{(3,3)} && {k^3 \xrightarrow{0} k^3} \\
    	{(2,2)} & \longleftrightarrow & {k^2 \xrightarrow{0} k^2} \\
    	{(1,1)} && {k \xrightarrow{0} k} \\
    	{(0,0)} && {0 \xrightarrow{} 0}
    	\arrow[no head, from=1-1, to=2-1]
    	\arrow[no head, from=1-3, to=2-3]
    	\arrow[no head, from=2-1, to=3-1]
    	\arrow[no head, from=2-3, to=3-3]
    	\arrow[no head, from=3-1, to=4-1]
    	\arrow[no head, from=3-3, to=4-3]
    	\arrow[no head, from=4-1, to=5-1]
    	\arrow[no head, from=4-3, to=5-3]
    \end{tikzcd}\]
    
    For $n = 2$, the poset of degenerations is isomorphic to $(\bbN^2,\leq)$:
    \[\begin{tikzcd}[sep=tiny]
    	\vdots && \vdots && \vdots && \vdots && \vdots \\
    	& {(3,3,0)} && {(2,3,1)} && {(1,3,2)} && {(0,3,3)} \\
    	&& {(2,2,0)} && {(1,2,1)} && {(0,2,2)} \\
    	&&& {(1,1,0)} && {(0,1,1)} \\
    	&&&& {(0,0,0)}
    	\arrow[no head, from=1-1, to=2-2]
    	\arrow[no head, from=1-3, to=2-2]
    	\arrow[no head, from=1-3, to=2-4]
    	\arrow[no head, from=1-5, to=2-4]
    	\arrow[no head, from=1-5, to=2-6]
    	\arrow[no head, from=1-7, to=2-6]
    	\arrow[no head, from=1-7, to=2-8]
    	\arrow[no head, from=1-9, to=2-8]
    	\arrow[no head, from=2-2, to=3-3]
    	\arrow[no head, from=2-4, to=3-3]
    	\arrow[no head, from=2-4, to=3-5]
    	\arrow[no head, from=2-6, to=3-5]
    	\arrow[no head, from=2-6, to=3-7]
    	\arrow[no head, from=2-8, to=3-7]
    	\arrow[no head, from=3-3, to=4-4]
    	\arrow[no head, from=3-5, to=4-4]
    	\arrow[no head, from=3-5, to=4-6]
    	\arrow[no head, from=3-7, to=4-6]
    	\arrow[no head, from=4-4, to=5-5]
    	\arrow[no head, from=4-6, to=5-5]
    \end{tikzcd}\]
    
    For $n\geq 3$, the poset of degenerations is no longer isomorphic to $(\bbN^{n-1}, \leq)$. Instead, for $n = 3$, we have a poset of the form:\\
    \includegraphics[width = \textwidth]{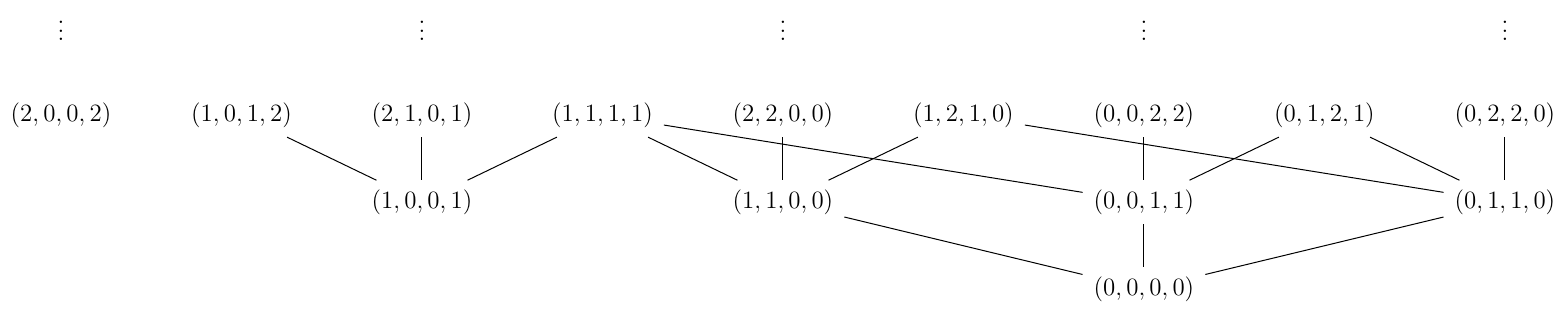}
    Minimal elements have the form $(n,0,0,n)$ for any $n \in \bbN$.
    
    We can also describe the degeneration order on $\sK^b_0$. The elements are given as vectors in $\bbN^\bbZ$ such that $\sum_{i \in \bbZ} (-1)^i a_{i} = 0$. The degenerations are still obtained from the same cover relations. The minimal elements are the $a$ such that $a_i \neq 0 \Rightarrow a_{i+1} = 0$. Note that there is no maximal element.
\end{example}

\subsection{Degenerations and triangles}
\label{subsect:deg_in_Kb}

In the following, we conflate points $M \in \varPhom$ with associated complexes up to homotopy and treat them as such.

\begin{theorem}
\label{thm:deg_traingles_homotopy}
    Let $M,N \in \varPhom$. The following are equivalent:
    \begin{enumerate}
        \item $M \leqdeghom N$,
        \item there exists a complex $Z \in \homprojstar$ and a triangle $ N \rightarrow M \oplus Z \rightarrow Z \rightarrow N[1]$,
        \item there exists a complex $Z \in \homprojstar$ and a triangle $ Z \rightarrow M \oplus Z \rightarrow N \rightarrow Z[1]$.
    \end{enumerate}
\end{theorem}

\begin{proof}
We will show that (1) is equivalent to (2). The connection to the triangles in (3) is obtained dually.

We first show (1) $\Rightarrow$ (2).
Let $M \leqdeghom N$. By \Cref{rmk:deg_homotopy}, there are acyclic complexes $Q_M$, $Q_N$ such that $M \oplus Q_M \leqdeg N \oplus Q_N$ in $\bchnproj$. So by \Cref{thm:deg_and_ses_in_chnproj} there exists a short exact sequence in $\bchnproj$
\begin{equation*}
    0 \rightarrow N \oplus Q_N \rightarrow M \oplus Q_M \oplus Z \rightarrow Z \rightarrow 0.
\end{equation*}
By \Cref{lem:triangles_and_ses}, this gives rise to a triangle in $\sK^{[-n,0]}(\proj \Lambda)$
\begin{equation*}
    N \oplus Q_N \rightarrow M \oplus Q_M \oplus Z \rightarrow Z \rightarrow (N \oplus Q_N)[1].
\end{equation*}
Since null-homotopic complexes are considered to be zero, this triangle is isomorphic to $N  \rightarrow M \oplus Z \rightarrow Z \rightarrow N[1]$.

Now we show (2) $\Rightarrow$ (1). Let us assume we have a triangle $N  \rightarrow M \oplus Z \rightarrow Z \rightarrow N[1]$. By \Cref{lem:triangles_and_ses}, there is an isomorphic triangle $N \rightarrow X \oplus Z \rightarrow Z \rightarrow N[1]$ such that $0 \rightarrow N \rightarrow X \rightarrow Z \rightarrow 0$ is a short exact sequence in $\bchnproj$. Since $M \oplus Z \cong X$ in $\sK^b(\proj\Lambda)$, there exist null-homotopic complexes $Q_M$ and $Q_X$ such that $M \oplus Z \oplus Q_M \cong X \oplus Q_X$ in $\sC^b(\proj\Lambda)$. Since $0 \rightarrow N \oplus Q_X \rightarrow X \oplus Q_X \rightarrow Z \rightarrow 0$ is still a short exact sequence and isomorphic to $0 \rightarrow N \oplus Q_X \rightarrow M \oplus Z \oplus Q_M \rightarrow Z \rightarrow 0$, we have that $M \oplus Q_M \leqdeg N \oplus Q_X$ in $\bchnproj$. In particular, $M \leqdeghom N$.
\end{proof}

\begin{example}
\label{ex:hom_deg_triangles}
The short exact sequences describing the covering relations in \Cref{ex:trivial_deg_ses} induce triangles describing the covering relations for $kA_1$ in \Cref{ex:A_1_homotopy}. 
\end{example}

We may apply this result to compare degenerations in $\ivarPhom{a}{b}$ and $\bvarPhom$.

\begin{theorem}
Let $[a,b] \subset [a',b']$. Let $M, N \in \ivarPhom{a}{b}$. They correspond to points $M,N \in \varPhom$ with $\ast \in \{[a',b'], \mathbf{b}\}$. If there is a degeneration $M \leqdeghom N$ in $\varPhom$ then there is a degeneration $M \leqdeghom N$ in $\ivarPhom{a}{b}$.
\end{theorem}

\begin{proof}
    By \Cref{thm:deg_traingles_homotopy}, there exists a complex $Z \in \homprojstar$ and a triangle (T1) $N \rightarrow M \oplus Z \rightarrow Z \rightarrow N[1]$. We will show that we can truncate $Z$ so that it is in $\homprojint{a}{b}$. Recall that there are two triangles (T2) $Z_{>b} \rightarrow Z \rightarrow Z_{\leq b} \rightarrow Z_{>b}[1]$ and (T3) $Z_{[a,b]} \rightarrow Z_{\leq b} \rightarrow Z_{<a} \rightarrow Z_{[a,b]}[1]$. We apply the octahedral axiom to (T1) and (T2). 
    \begin{equation*}
        \begin{tikzcd}
        	& {Z_{>b}} & {Z_{>b}} & \\
        	N & {M \oplus Z} & Z & {N[1]} \\
        	N & X & {Z_{\leq b}} & {N[1]} \\
        	& {Z_{>b}[1]} & {Z_{>b}[1]}
        	\arrow[equals, from=1-2, to=1-3]
        	\arrow["\varphi"', from=1-2, to=2-2]
        	\arrow[from=1-3, to=2-3]
        	\arrow[from=2-1, to=2-2]
        	\arrow[equals, from=2-1, to=3-1]
        	\arrow[from=2-2, to=2-3]
        	\arrow[from=2-2, to=3-2]
        	\arrow[from=2-3, to=2-4]
        	\arrow[from=2-3, to=3-3]
        	\arrow[equals, from=2-4, to=3-4]
        	\arrow[from=3-1, to=3-2]
        	\arrow[from=3-2, to=3-3]
        	\arrow[from=3-2, to=4-2]
        	\arrow[from=3-3, to=3-4]
        	\arrow[from=3-3, to=4-3]
        	\arrow[equals, from=4-2, to=4-3]
        \end{tikzcd}
    \end{equation*}
    Since $M$ has no entries in degree bigger than $b$, we know that $\varphi = (0,\varphi_Z)^T$. So we can conclude that $X = {M \oplus Z_{\leq b}}$. So we obtain a triangle (T4) $N \rightarrow M \oplus Z_{\leq b} \rightarrow Z_{\leq b} \rightarrow N[1]$. Now we apply the octahedral axiom to (T3) and (T4).
    \[\begin{tikzcd}
    	& {Z_{<a}[-1]} & {Z_{<a}[-1]} & \\
    	N & Y & {Z_{[a,b]}} & {N[1]} \\
    	N & {M \oplus Z_{\leq b}} & {Z_{\leq b}} & {N[1]} \\
    	& {Z_{<a}} & {Z_{<a}}
    	\arrow[equals, from=1-2, to=1-3]
    	\arrow[from=1-2, to=2-2]
    	\arrow[from=1-3, to=2-3]
    	\arrow[from=2-1, to=2-2]
    	\arrow[equals, from=2-1, to=3-1]
    	\arrow[from=2-2, to=2-3]
    	\arrow[from=2-2, to=3-2]
    	\arrow[from=2-3, to=2-4]
    	\arrow[from=2-3, to=3-3]
    	\arrow[equals, from=2-4, to=3-4]
    	\arrow[from=3-1, to=3-2]
    	\arrow[from=3-2, to=3-3]
    	\arrow["\psi"', from=3-2, to=4-2]
    	\arrow[from=3-3, to=3-4]
    	\arrow[from=3-3, to=4-3]
    	\arrow[equals, from=4-2, to=4-3]
    \end{tikzcd}\]
    Again, we have that $\psi = (0, \psi_Z)$ so $Y = M \oplus Z_{[a,b]}$. So we obtain a triangle $N \rightarrow M \oplus Z_{[a,b]} \rightarrow Z_{\leq b} \rightarrow N[1]$ satisfying \Cref{thm:deg_traingles_homotopy}. Thus there is a degeneration $M \leqdeghom N$ in $\ivarPhom{a}{b}$.
\end{proof}

For another application, we can show a link from module degenerations to chain complex degenerations in $\sK^*_g$.

\begin{lemma}
\label{lem:module_deg_to_cx_deg}
    Let $\mathrm{gldim}(\Lambda) < \infty$ and $M,N \in \mod \Lambda$ with minimal projective resolutions $P_M^{\bullet}, P_N^{\bullet}$. If $M \leqdeg N$ then $P_M^{\bullet} \leqdeghom P_N^{\bullet}$.
\end{lemma}

\begin{proof}
    This follows directly from the horseshoe lemma using that projective resolutions are unique up to homotopy.
\end{proof}

\begin{example}
   We return to $kA_2$ and \Cref{ex:A_2}. We consider the representation variety for $\dd = (1,1)$. Clearly, it is of the form $\rep(\Lambda, \dd) = k$. It has two orbits: the semisimple one containing only $S_1 \oplus S_2$ and the orbit of $P_1$. We note that $P_1$ degenerates to $S_1 \oplus S_2$ similarly to \Cref{ex:trivial_deg}. The minimal projective resolution of $S_1$ is $P_{S_1} \colon P_2 \rightarrow P_1$, the one of $S_2 = P_2$ is $P_{S_2} \colon 0 \rightarrow P_2$ and the one of $P_1$ is $P_{P_1} \colon 0 \rightarrow P_1$. So by \Cref{lem:module_deg_to_cx_deg}, $P_{P_1} \leqdeghom P_{S_1} \oplus P_{S_2}$. This degeneration is related to $\pmat{0 \amsamp 0 \\ 0 \amsamp 1} \leqdeg \pmat{0 \amsamp 1 \\ 0 \amsamp 0}$ in \Cref{ex:A_2}.
\end{example}

\begin{remark}
    In general, for complexes $X, Y \in \sK_g^*$ with $X \leqdeghom Y$, the associated triangle $Y \rightarrow X \oplus Z \rightarrow Z \rightarrow Y[1]$ gives rise to a long exact sequence of its cohomologies. However, this long exact sequence does not in general split into short exact sequences that would easily give rise to module degenerations of the cohomologies. 
\end{remark}

\subsection{The ext and hom order}

We can also define analogues of the ext and hom order for $\bhomproj$. 

\begin{definition}
\label{def:ext_order_hom}
    Let $X,Y \in \bhomproj$. We say $X \leqexthom Y$ if there exist an integer $n \geq 1$ and a chain of $n$ triangles $ A_i \rightarrow B_i \rightarrow C_i \rightarrow C_i[1]$ with $B_1 = X$, for every $i \geq 2$, $B_i = A_{i-1} \oplus C_{i-1}$ and $Y = A_n \oplus C_n$.
\end{definition}

\begin{remark}
    This is a well defined order on isomorphism classes of $\bhomproj$. Note that similarly to the degeneration order, this is the colimit of the extension order of chain complexes in the sense that $X \leqexthom Y$ if and only if there exist acyclic complexes $N_X, N_Y$ such that $X \oplus N_X \leqext Y \oplus N_Y$.
\end{remark}

\begin{definition}
\label{def:hom_order_hom}
    Let $X,Y \in \bhomproj$. We say $X \leqhomhom Y$ if there exist acyclic complexes $N_X, N_Y$ such that $X \oplus N_X \leqhom Y \oplus N_Y$. 
\end{definition}

\begin{remark}
    This is a well defined order on isomorphism classes of $\bhomproj$ since for any acyclic complex $Q$, $X \leqhom Y \Rightarrow X \oplus Q \leqhom Y \oplus Q$.
\end{remark}

The following implications follow directly from \Cref{lem:ext_deg_hom_ch_cx}.

\begin{lemma}
\label{lem:ext_deg_hom_hom}
    Let $X,Y \in \bhomproj$. Then the following implications hold:
    \begin{equation*}
        (X \leqexthom Y) \Rightarrow (X \leqdeghom Y) \Rightarrow (X \leqhomhom Y).
    \end{equation*}
\end{lemma}

\begin{example}
\label{ex:ext_not_deg_Kb}
    The ext and degeneration order are not in general identical.  We consider an example stemming from representation varieties in \cite{Rie}. Let 
    \begin{equation*}
        \Lambda = k(
        \begin{tikzcd}
        	1 & 2
        	\arrow["a", from=1-1, to=1-2]
        	\arrow["b", from=1-2, to=1-2, loop, in=325, out=35, distance=10mm]
        \end{tikzcd})/\langle b^2\rangle.
    \end{equation*}
    We define two modules $M$ and $N$ as follows:
    \begin{equation*}
        M = \begin{tikzcd}
        	k & k
        	\arrow["\pmatquiver{1 \\ 0}", from=1-1, to=1-2]
        	\arrow["\pmatquiver{0 \amsamp 0 \\ 1 \amsamp 0}", from=1-2, to=1-2, loop, in=325, out=35, distance=10mm]
        \end{tikzcd}
        \qquad
        N = \begin{tikzcd}
        	1 & 2
        	\arrow["\pmatquiver{0 \\ 1}", from=1-1, to=1-2]
        	\arrow["\pmatquiver{0 \amsamp 0 \\ 1 \amsamp 0}", from=1-2, to=1-2, loop, in=325, out=35, distance=10mm]
        \end{tikzcd}
    \end{equation*}
    Then $M$ degenerates to $N$ in the representation variety. Since $M \cong P_1$ is projective, its minimal projective resolution is $P_M^\bullet = (0 \rightarrow P_1)$. The minimal projective resolution of $N$ is $P_N^\bullet = (P_2 \rightarrow P_1 \oplus P_2)$. By \Cref{lem:module_deg_to_cx_deg}, $P_M^\bullet \leqdeghom P_N^\bullet$ but this relation cannot be obtained in the ext order since $P_N^\bullet$ is indecomposable.
\end{example}

\biblio

\section*{Acknowledgements}

This work is part of the author's PhD. The author thanks her advisors Claire Amiot and Pierre-Guy Plamondon and for their direction and insights. 

Part of this work was done while the author was at the Université du Québec à Montréal. This mobility was funded by the Mitacs Globalink Research Award. The author would like to thank Hugh Thomas for helpful discussions and the Laboratoire d'Algèbre, de Combinatoire et d’Informatique Mathématique for a warm welcome.

The author also thanks René Marczinzik, Calvin Pfeiffer and Jan Schröer for interesting discussions.

This project has received funding from the European Union’s Horizon Europe research and innovation programme under the Marie Skłodowska-Curie grant agreement No 101126554.

\section*{Disclaimer}
Co-Funded by the European Union. Views and opinions expressed are however those of the author only and do not necessarily reflect those of the European Union. Neither the European Union nor the granting authority can be held responsible for them.
\noindent \includegraphics[height = 1cm]{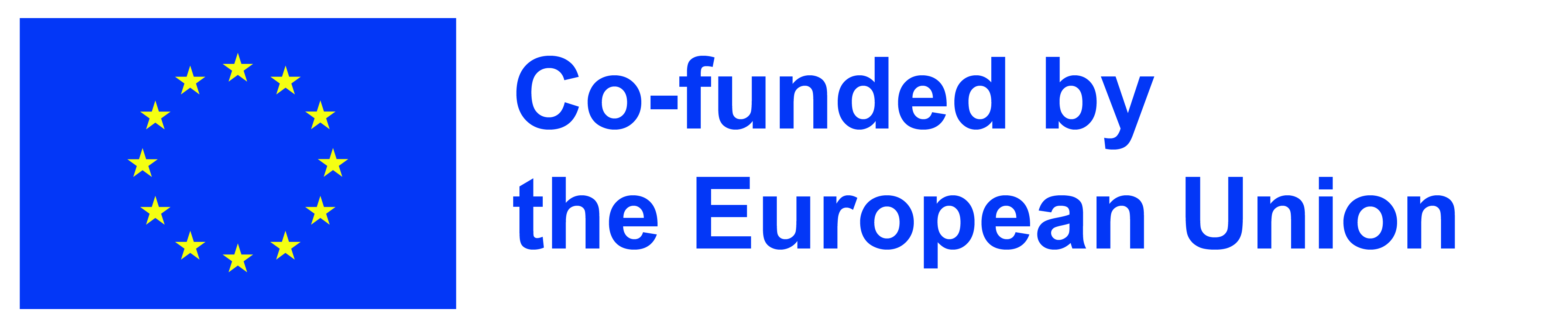}

\biblio

\bibliographystyle{amsalpha}
\bibliography{literature}

\end{document}